\documentclass[final,1p,times,number]{elsarticle}

\usepackage[pdftex,colorlinks,citecolor=blue,urlcolor=blue]{hyperref}
 \usepackage{blkarray}
\usepackage{graphicx,color,tikz,pgf}
\usepackage[ruled]{algorithm2e}
\usepackage{comment}

\graphicspath{{./Figures_Indep_Model/}}
\usepackage{array}
\usepackage{amsmath}
\usepackage{amsthm}
\usepackage{amssymb}
\usepackage{amsfonts}
\usepackage{cleveref}

\newtheorem{theorem}{Theorem}
\newtheorem{lemma}[theorem]{Lemma}
\newtheorem{corollary}[theorem]{Corollary}

\newtheorem{proposition}[theorem]{Proposition}

\newtheorem{example}[theorem]{Example}

\newcommand{\conv}{\operatorname{conv}}

\numberwithin{equation}{section}

\begin{document}
	
\begin{frontmatter}
	
	\title{Wasserstein Distance to Independence Models}
	
	\author{T\"urk\"u \"Ozl\"um \c{C}el{\.i}k}
	\address{Simon Fraser University, 8888 University Drive, Burnaby, Canada}
	\ead{turkuozlum@gmail.com}
	
	\author{Asgar Jamneshan}
	\address{UCLA, 520 Portola Plaza, Los Angeles, USA}
	\ead{jasgar@math.ucla.edu}
	
	\author{Guido Mont\'ufar}
	\address{MPI-MiS Leipzig, Inselstr. 22, Leipzig, Germany and UCLA, 520 Portola Plaza, Los Angeles, USA}
	\ead{guido.montufar@mis.mpg.de}
	
	\author{Bernd Sturmfels}
	\address{MPI-MiS Leipzig, Inselstr. 22, Leipzig, Germany and UC Berkeley, 970 Evans Hall, Berkeley, USA}
	\ead{bernd@mis.mpg.de}
	
	\author{Lorenzo Venturello}
	\address{Department of Mathematics, KTH Royal Institute of Technology, Stockholm,
	 Lindstedtsv\"agen 25, Stockholm, Sweden}
		\ead{lorenzo.venturello@hotmail.it}

	\begin{abstract}
		An independence model for discrete random variables is a Segre-Veronese variety in a probability simplex. Any metric on the set of joint states of the random variables induces a Wasserstein metric on the probability simplex. The unit ball of this polyhedral norm is dual to the Lipschitz polytope. Given any data distribution, we seek to minimize its Wasserstein distance to a fixed independence model. The solution to this optimization problem is a piecewise algebraic function of the data. We compute this function explicitly in small instances, we study its combinatorial structure and algebraic degrees in general, and we present some experimental case~studies.
	\end{abstract}
	
\end{frontmatter}	
%
%
%

\section{Introduction}

A probability distribution on the finite set $[n] = \{1,2,\ldots,n\}$ is a point $\nu$ in the simplex $\Delta_{n-1}=\{(\nu_1,\ldots,\nu_n) \in \mathbb{R}^n_{\geq 0}: \sum_{i=1}^n \nu_i=1\}$. 
We metrize this simplex by the {\em Wasserstein distance}. To define this,
we first turn the state space $[n]$ into a  metric space by fixing a symmetric $n\times n$ matrix 
$d=(d_{ij})$ with nonnegative entries. These  satisfy
$d_{ii} =0$ and $d_{ik} \leq d_{ij} + d_{jk}$ for all $i,j,k$.

Given two probability distributions 
$\mu,\nu\in\Delta_{n-1}$, we 
consider the following linear programming problem,
where $x = (x_1,\ldots,x_n)$ denotes the decision variables: 
\begin{equation}
\label{eq:dual2}   
{\rm Maximize} \,\,\,
\sum_{i=1}^n\, (\mu_i-\nu_i) \,x_i \, \,\,\,
\text{subject to}\,\,\,\, |x_i - x_j| \,\leq\, d_{ij} \,\,\,\,
\text{for all} \,\, \,1\leq i < j \leq n.
\end{equation}
The optimal value of \eqref{eq:dual2} is denoted $W_d(\mu,\nu) $ 
and called the {\em Wasserstein distance} between $\mu$ and~$\nu$.
This is a metric on  $\Delta_{n-1}$ induced from the finite metric space $([n],d)$. 
The linear program \eqref{eq:dual2} is known as the 
{\em Kantorovich dual} of the  {\em optimal transport problem}
\cite{BASSETTI20061298,
	villani08}. 
In \cite{macis}, we emphasized the optimal transport perspective, whereas here we prefer the dual formulation \eqref{eq:dual2}.

The feasible region of the linear program \eqref{eq:dual2} is unbounded since it is invariant under
translation by ${\bf 1} = (1,1,\ldots,1)$. Taking the
quotient modulo the line $\mathbb{R} {\bf 1}$, we obtain the compact set
\begin{equation}
\label{eq:polytope}
 P_d \,\, = \,\, \bigl\{\,x \in \mathbb{R}^n /\mathbb{R}{\bf 1} \,\,:\,\,\, |x_i - x_j| \,\leq\, d_{ij} \,\,\,\,
\text{for all} \,\, \,1\leq i < j \leq n \,\bigr\}. 
\end{equation}
This $(n-1)$-dimensional polytope is the
{\em Lipschitz polytope} of the metric space $([n],d)$.
In tropical geometry  \cite{JK, Ngoc}, one refers to
$P_d$ as a {\em polytrope}. It is convex both classically and tropically.

An optimal solution $x^*\in P_d$ to the problem  \eqref{eq:dual2} is 
 an {\em optimal discriminator} for the two probability distributions $\mu$ and $\nu$. 
It satisfies $\, W_d(\mu,\nu) = \langle \mu - \nu, x^* \rangle$.
Its coordinates $x^*_i$ are weights on the state space $[n]$ that tell  $\mu $ and $\nu$  apart. 
Here $\langle \,\cdot\,,\, \cdot\,\rangle$ is the standard inner product on $\mathbb{R}^n$.

\smallskip

In this article, we study the Wasserstein distance from a distribution
$\mu$ to a fixed {\em discrete statistical model}  $\,\mathcal{M} \subset \Delta_{n-1}$. We consider the case where $\mathcal{M}$ is a compact set defined by polynomial constraints on $\nu_1,\ldots,\nu_n$.  
Our task is to solve the following mini-max optimization problem:
\begin{equation}
\label{eq:ourproblem}
  W_d(\mu,\mathcal{M})\quad :=\quad \min_{\nu\in \mathcal{M}} 
  W_d(\mu,\nu) \quad
   = \quad \min_{\nu\in \mathcal{M}}\,\max_{x\in P_d}\, \langle \mu-\nu, x \rangle.
\end{equation}
Computing this quantity means solving a non-convex optimization problem. 
We study this problem and propose solution strategies, using
methods from geometry, algebra and combinatorics.
The analogous problem for the Euclidean metric
was treated in \cite{DHOST} and various subsequent works.

The term independence model in our title refers to a statistical model for $k$ discrete random variables where the state space is the product $[m_1] \times \cdots \times [m_k]$ and the $m_i$ are positive integers. 
The number of states equals $n= m_1 \cdots m_k$. 
The simplex $\Delta_{n-1}$ consists of all tensors $\nu$ of format $m_1 \times \cdots \times m_k$ with nonnegative entries that sum to $1$. 
The {\em independence model} $\mathcal{M}$ is the subset of tensors $\nu$ that have rank one.
These represent joint distributions for $k$ independent discrete random variables. 
Recall that a tensor has {\em rank one} if it can be written as an outer product of vectors of sizes $m_1,\ldots, m_k$.  In algebraic geometry, the model $\mathcal{M}$ is known as the {\em Segre variety}. 
Of particular interest is the case $m_1 = \cdots = m_k = 2$ for which $\mathcal{M}$ is the {\em  $k$-bit independence model}. 

We also consider independence models for symmetric tensors.  Here,
all $k$ random variables share the same marginal distribution, so
the number of states is $n = \binom{m+k-1}{k}$ where
$m := m_1 = \cdots = m_k$.
The model $\mathcal{M}$ of symmetric tensors of rank one is the
{\em Veronese variety}. The definition of independence
by way of rank one tensors generalizes
to many other settings. For instance, one may consider
 partially symmetric tensors, when $\mathcal{M}$ is a 
{\em Segre-Veronese variety} (cf.~\cite[\S 8]{DHOST}).

Let us restate our problem for joint distributions. Given an arbitrary tensor $\mu \in \Delta_{n-1}$, we seek an independent tensor $\nu \in \mathcal{M}$ that is closest to $\mu$ with respect to the Wasserstein distance $W_d$. One natural choice for the underlying metric $d$ is the Hamming distance on  strings in $[m_1] \times \cdots \times [m_k]$.  
We consider various metrics in this paper. While the analysis in Section \ref{sec3} is carried out for general finite metric spaces, we consider three types of metrics relevant in applications for the combinatorial analysis in Section \ref{sec4}, namely the discrete metric, the $L_0$-metric, and the $L_1$-metric.     

Our approach  centers around
the {\em optimal value function} $\,\mu \mapsto W_d(\mu,\mathcal{M})\,$
and the {\em solution function} 
$\mu \mapsto {\rm argmin}_{\nu\in\mathcal{M}} \,W_d(\mu,\nu)$. 
The latter is multivalued since there can be two or more
optimal solutions for  special $\mu$.
The guiding idea is to find algebraic formulas for these functions.
We will demonstrate this in Section~\ref{sec2}  with
explicit results for the two smallest instances, with $k=m=2$ and fixed $d$. This rests on a geometric study in the triangle $\Delta_2$ of symmetric $2 \times 2$ matrices, and in the tetrahedron $\Delta_3$ of all $2 \times 2$ matrices, with nonnegative entries that sum to $1$. 

The optimal value function and the solution function are piecewise algebraic. 
This suggests a division of our problem into two tasks: first
identify all pieces, then find a formula for each piece. 
This will be explained in Section~\ref{sec3} where we review 
basics regarding polyhedral norms and characterize the geometry of the distance function to an algebraic variety under such a norm. 

Both tasks are characterized by a high degree of complexity.
 The first task pertains to {\em combinatorial complexity}. 
 This will be addressed in Section~\ref{sec4} with a  combinatorial study of the Lipschitz polytopes that are
associated with product state spaces like those of independence models. 
The second task pertains to {\em algebraic complexity}. This is our topic in Section~\ref{sec5}. We 
relate the algebraic degrees of the optimal value function to polar classes of the underlying model. We discuss 
and apply the formulas derived by \cite{Luca} for polar classes of  Segre-Veronese varieties.

Many optimization problems arising in the mathematics of data
involve both discrete and continuous structures.
In our view, it is important to separate these two, in order to clearly
understand the different mathematical features that arise.
In a setting like the one studied here, it is natural to  separate the combinatorial
complexity and the algebraic complexity of an optimization problem.
The former arises from the exponentially  many combinatorial  patterns, here  the
faces of a polytope,  one might see in a solution. The latter refers to the problem of solving
a system of polynomial equations, and the 
algebraic degree that is intrinsically associated with that task.

Consider the problem of minimizing the $L_\infty$-distance 
from a data point in $3$-space to a general cubic surface.
The optimal point on the surface is tangent to an $L_\infty$-ball
around the data point.
Each $L_\infty$-ball is a cube, just like in Figure \ref{fig:proof_2.2}.
 This tangency occurs
at either a vertex or an edge or a facet.
Thus the combinatorial complexity is given by the face numbers, $f = (8,12,6)$.
Every face determines  a system of polynomial equations
in three unknowns that the  optimal point satisfies.
The algebraic complexity is the expected number of complex solutions.
These numbers are the polar degrees, given by the vector
$\delta = (3,6,12)$ for cubic surfaces.
In Sections \ref{sec4} and \ref{sec5}, we compute the
vectors $f$ and $\delta$ for Wasserstein distance to the independence models.
Section~\ref{sec6} features numerical experiments. We solve our optimization problem
for a range of  instances using the software {\tt SCIP} \cite{SCIP}, and we discuss the
geometric insights that were learned.

\section{Explicit Formulas}
\label{sec2}

In this section, 
we solve
our problem for two binary random variables. We begin with the case of a  binomial distribution, namely the sum of two independent and identically distributed binary random variables. 
The model $\mathcal{M}$ is a quadratic curve in the probability triangle $\Delta_2$, known
among statisticians and biologists as the {\em Hardy-Weinberg curve}. This curve is the image of the~map
\begin{equation}
\label{eq:hardyweinberg}	\varphi\,:\,[0,1]\to\,\Delta_2\,, \quad
p\,\mapsto \, \bigl( \,p^2,\,2p(1-p),\,(1-p)^2 \,\bigr). 
\end{equation}	
Thus, $\mathcal{M}$ is the set of  nonnegative symmetric rank one matrices 
$\begin{pmatrix} \nu_1 \!\! &\! \tfrac12\nu_2 \\ \tfrac12\nu_2 \!\! & \!  \nu_3 \end{pmatrix} $
with $\nu_1+\nu_2 +\nu_3 = 1$.

Our second ingredient is the choice of a metric
$ d= (d_{12},d_{13},d_{23})$ on the state space $[3] = \{1,2,3\}$.
There are two natural choices: the {\em discrete metric}
$ d = (1,1,1)$ and the {\em $L_1$-metric} $d = (1,2,1)$. 
Their 
corresponding balls are 
illustrated in Figure~\ref{fig:wassersteinball}.
 Their optimal value functions agree,
so Theorem~\ref{thm: Hardy-Weinberg} is valid for both metrics. 
This holds only in such a small example. 
For larger independence models on symmetric tensors, these two metrics will lead to different solutions. 
	
\begin{figure}[h]
	\centering
	\includegraphics[scale=0.6]{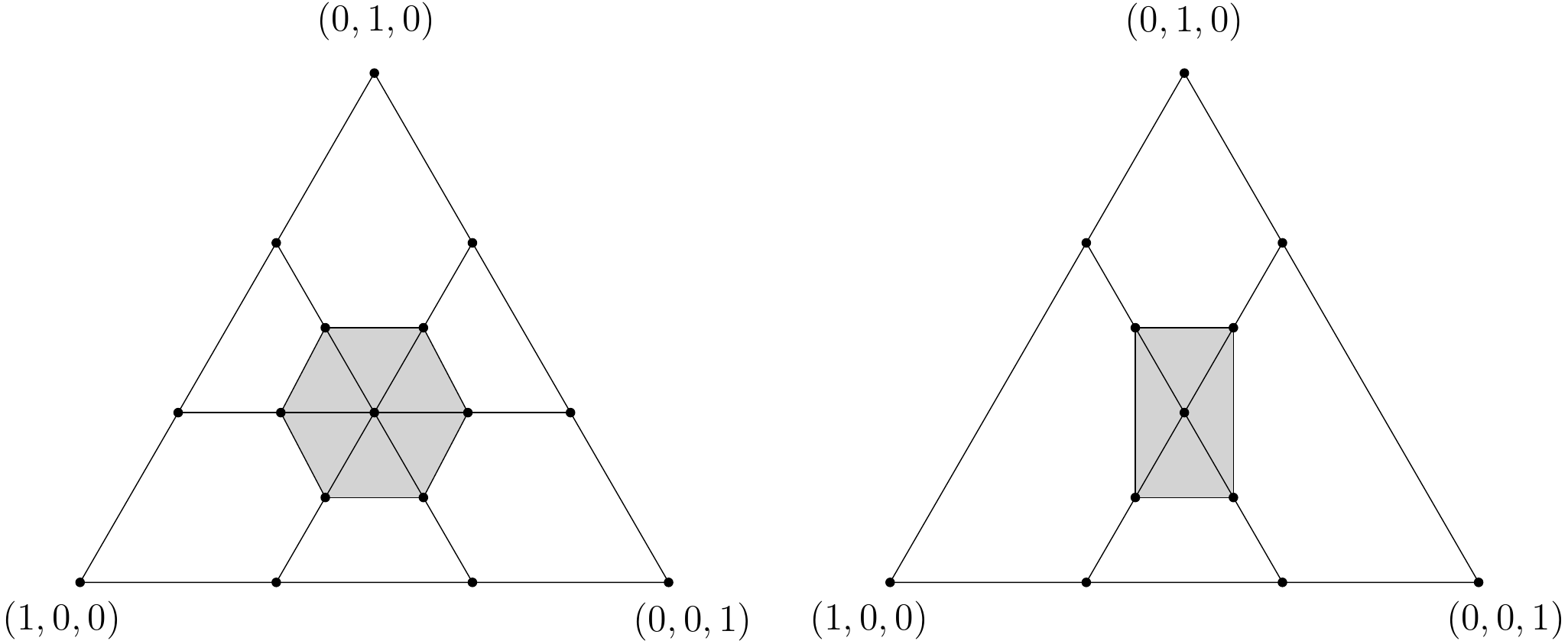}
	\caption{The Wasserstein balls of radius $\frac{1}{6}$ centered in the uniform distribution $(\frac{1}{3},\frac{1}{3},\frac{1}{3})$ associated to the discrete metric (left) and 
	the $L_1$-metric (right) for $n=3$. }
\label{fig:wassersteinball}	
\end{figure}

We now present the optimal value function and the solution function for the model in \eqref{eq:hardyweinberg}. 
These two functions are piecewise algebraic. 
The five pieces are shown in Figure~\ref{fig:fivepieces}. 
On four of them,~the solution function is algebraic of degree two. 
The formula involves a square root in the data~distribution. 
On the fifth piece, the solution function is constant and the optimal value function is linear. 

\begin{theorem}\label{thm: Hardy-Weinberg}
For the discrete metric and for the $L_1$-metric on the state space $[3] = \{1,2,3\}$,
 the Wasserstein distance from a data distribution $\mu \in \Delta_2$ to the 
Hardy-Weinberg curve $\mathcal{M}$ equals
	\[W_d(\mu,\mathcal{M}) \,= \, \begin{cases}
|2\sqrt{\mu_1}-2\mu_1-\mu_2| &\text{if }\quad\mu_1-\mu_3\geq 0 \text{ and } \mu_1\geq \frac{1}{4} ,\\
|2\sqrt{\mu_3}-2\mu_3-\mu_2| &\text{if }\quad \mu_1-\mu_3\leq 0 \text{ and } \mu_3\geq \frac{1}{4} ,\\
	\mu_2-\frac{1}{2} & \text{if }\quad \mu_1\leq \frac{1}{4} \text{ and } \mu_3\leq \frac{1}{4} .
	\end{cases}		
	\]
	The solution function
	 $ \,\Delta_2 \rightarrow \mathcal{M},\, \mu \mapsto \nu^*(\mu)\,$ 
	 is given (with the same case distinction) by 
	\[\nu^*(\mu)\,=\,\begin{cases}
	(\mu_1,2\sqrt{\mu_1}-2\mu_1, 1+\mu_1-2\sqrt{\mu_1}) , \\
	(1+\mu_3-2\sqrt{\mu_3},2\sqrt{\mu_3}-2\mu_3,\mu_3) ,\\
	(\frac{1}{4},\frac{1}{2},\frac{1}{4}) .
	\end{cases} 	
	\]	
	\end{theorem}

\vspace{-1mm}
\begin{figure}[h!]
	\centering	
	\includegraphics[scale=0.65]{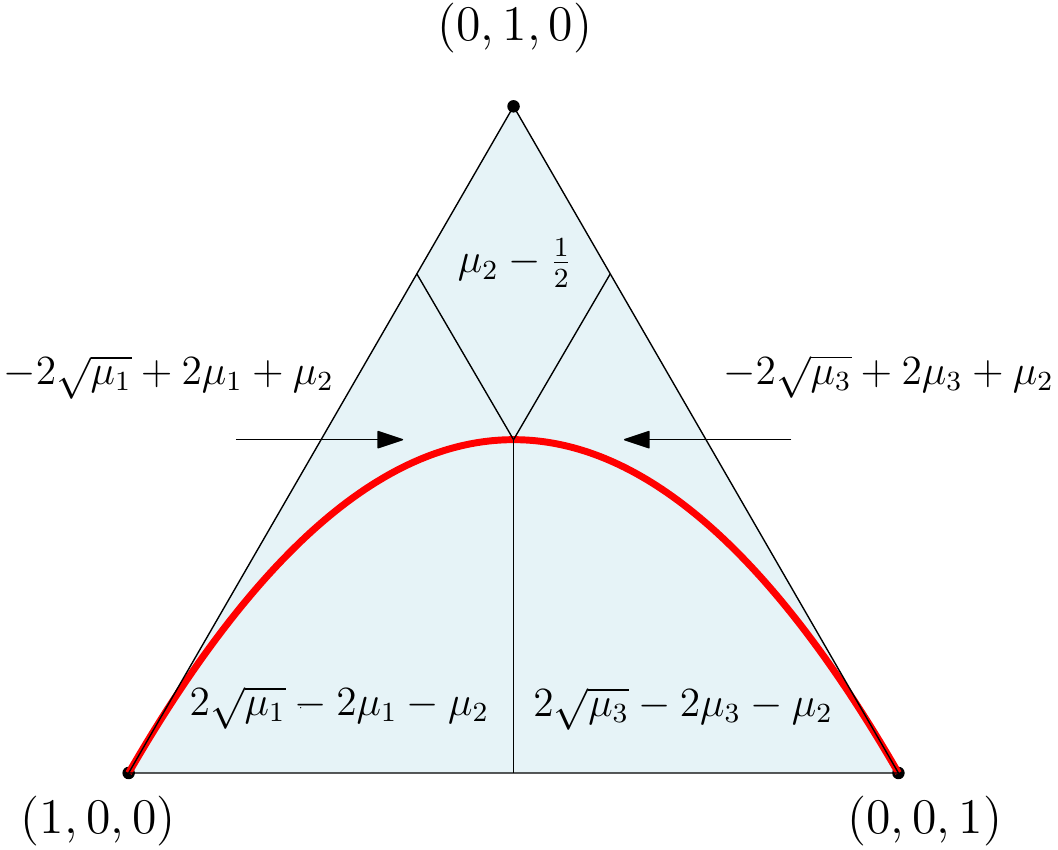}
	\caption{The Hardy-Weinberg curve $\mathcal{M}$ is shown in red.
	The optimal value function for the Wasserstein distance to this curve is piecewise algebraic with five regions.}
	\label{fig:fivepieces}	
\end{figure}

Theorem \ref{thm: Hardy-Weinberg}
involves a distinction into three cases.
Each of the first two cases gives two algebraic pieces of the optimal value function. 
We point out three interesting features.
First, there is a full-dimensional region in $\Delta_2$,
namely the top parallelogram in Figure \ref{fig:fivepieces},
all of whose points $\mu$ share the same optimal solution $\nu^*(\mu)=(\frac{1}{4},\frac{1}{2},\frac{1}{4})$ in $\mathcal{M}$. 
Second, all points $\mu$ on the vertical line segment $\{\mu: \mu_1=\mu_3, \mu_2< 1/2\}$ have two distinct optimal solutions, namely the intersection points of the curve $\mathcal{M}$ with a horizontal line. The identification of such \emph{walls of indecision} is important for finding accurate numerical solutions. 
Third, the optimal value and solution functions agree for the two metrics in Figure~\ref{fig:wassersteinball}. However, one can perturb the discrete metric to observe a difference. This is illustrated in Figure \ref{fig:lorenzonew}. 
The point $\mu = (\frac{1}{2},0,\frac{1}{2})$ has
two closest points in the $L_1$-metric but four
closest points in the Wasserstein distance induced by  $d=(d_{12},d_{13},d_{23})=(1,1-\epsilon,1)$ for some $\epsilon>0$. 

\begin{figure}[h]
	\centering
	\includegraphics[scale=0.55]{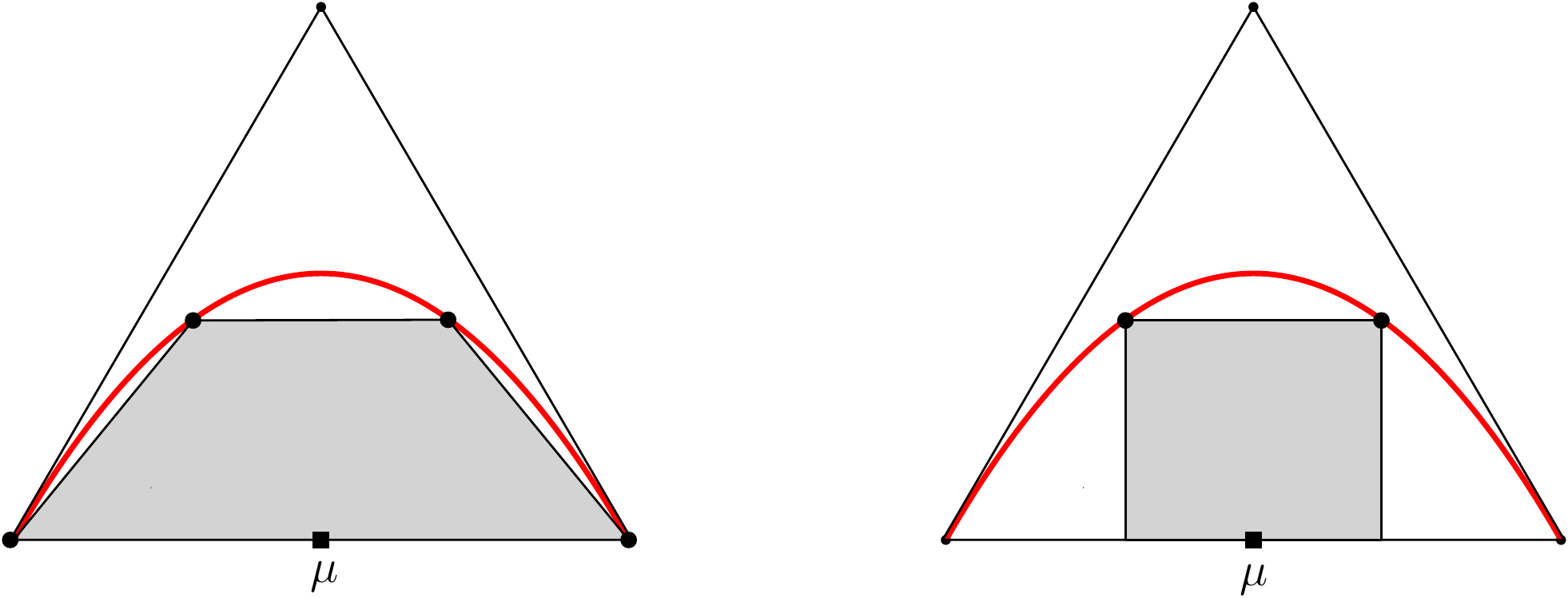} \vspace{-0.1in}
	\caption{The Wasserstein balls around a data point touch the curve in either 
	four or two points. The metrics on $[3]$ are $d=(1,1-\epsilon,1)$ and
	$d=(1,2,1)$ respectively.}
\label{fig:lorenzonew}	
\end{figure}

Next, we increase the dimension by one. Consider
the tetrahedron $\Delta_3$ whose points are joint probability
distributions of two binary random variables $(n=4,k=2)$.
The {\em $2$-bit independence model} $\mathcal{M}\subset \Delta_3$ consists of all nonnegative $2 \times 2$ matrices of rank one whose entries sum to one: 
\begin{equation}
\label{eq:model}
 \qquad \begin{pmatrix} \nu_1 & \nu_2 \\ \nu_3 & \nu_4  \end{pmatrix}
\, = \,
\begin{pmatrix} pq & p(1-q) \\ (1{-}p)q & (1{-}p)(1{-}q) \end{pmatrix} \! , \qquad
\,  (p,q)\in[0,1]^2. 
\end{equation}
Thus, $\mathcal{M}$ is the  surface in the tetrahedron $\Delta_3$
 defined by the equation $\nu_1 \nu_4 = \nu_2 \nu_3 $.
 We fix the $L_0$-metric $d$ on the set of binary pairs $[2]\times[2]$. 
 Under our identification (lexicographic order) of this state space with $[4] = \{1,2,3,4\}$, the resulting metric on $\Delta_3$ is given by the $4 \times 4$~matrix 
   \begin{equation}
\label{eq:metric}
d\,\,=\,\,
\begin{small}
\begin{pmatrix}
\,0 & 1 & 1 & 2    \\
\,1 & 0 & 2 & 1    \\
\,1 & 2 & 0 & 1   \\
\,2 & 1 & 1 & 0  \\
\end{pmatrix}. 
\end{small}
\end{equation}

We now present the optimal value function
and solution function for this independence~model.

\begin{theorem}\label{thm: 3bits}
For the $L_0$-metric on the state space $[2]\times[2]$, the Wasserstein distance from a data 
distribution $\mu \in \Delta_3$ to the $2$-bit independence surface $\mathcal{M}$ is given by
	\[W_d(\mu,\mathcal{M})\,\,=\,\,\begin{cases}
	2\sqrt{\mu_1}(1-\sqrt{\mu_1})-\mu_2-\mu_3 &\text{if } \mu_1 \geq \mu_4\,, \,\,
	\sqrt{\mu_1} \geq \mu_1+\mu_2\,,\,\,
		\sqrt{\mu_1} \geq \mu_1+ \mu_3, \smallskip \\
	2\sqrt{\mu_2}(1-\sqrt{\mu_2})-\mu_1-\mu_4 &\text{if } \mu_2 \geq \mu_3\,, \,\,
	\sqrt{\mu_2} \geq \mu_1+\mu_2\,,\,\, \sqrt{\mu_2} \geq \mu_2+\mu_4, \smallskip \\
	2\sqrt{\mu_3}(1-\sqrt{\mu_3})-\mu_1-\mu_4 &\text{if } \mu_3 \geq \mu_2\,, \,\,
	\sqrt{\mu_3} \geq \mu_1+\mu_3\,,\,\, \sqrt{\mu_3} \geq \mu_3+\mu_4, \smallskip \\
	2\sqrt{\mu_4}(1-\sqrt{\mu_4})-\mu_2-\mu_3 &\text{if } \mu_4 \geq \mu_1\,,\,\,
	 \sqrt{\mu_4} \geq \mu_2+\mu_4 \,,\,\, \sqrt{\mu_4} \geq \mu_3+\mu_4, \smallskip \\
	|\mu_1\mu_4-\mu_2\mu_3|/(\mu_1+\mu_2)  &\text{if } \mu_1 \geq \mu_4 , \,
	 \mu_2\geq \mu_3 ,  \, 	\mu_1{+}\mu_2 \geq \sqrt{\mu_1},\,
\mu_1{+}\mu_2 \geq\sqrt{\mu_2}, \smallskip \\
	|\mu_1\mu_4-\mu_2\mu_3|/(\mu_1+\mu_3) &\text{if } \mu_1 \geq \mu_4,  \,
	\mu_3\geq \mu_2, \,\mu_1 {+} \mu_3 \geq	\sqrt{\mu_1},\,
	\mu_1{+}\mu_3	\geq \sqrt{\mu_3},	\smallskip	\\
	|\mu_1\mu_4-\mu_2\mu_3|/(\mu_2+\mu_4) &\text{if } \mu_4 \geq \mu_1, \,
	 \mu_2\geq \mu_3, \, \mu_2{+}\mu_4 \geq \sqrt{\mu_4},\,
	\mu_2{+}\mu_4 \geq  \sqrt{\mu_2}	, \smallskip \\
	|\mu_1\mu_4-\mu_2\mu_3|/(\mu_3+\mu_4) &\text{if }\mu_4 \geq \mu_1, \,
	 \mu_3\geq \mu_2, \,
	\mu_3{+}\mu_4 \geq   \sqrt{\mu_4},\,
	\mu_3{+}\mu_4 \geq  \sqrt{\mu_3}.
	\end{cases}	
	\]
	The solution function
	 $ \,\Delta_3 \rightarrow \mathcal{M},\, \mu \mapsto \nu^*(\mu)\,$ 
	 is given (with the same case distinction) by 
	\[\nu^*(\mu)\,\,= \,\,\begin{cases}
	\bigl( \,\mu_1\, , \,\,\sqrt{\mu_1}-\mu_1\, , \,\,\sqrt{\mu_1}-\mu_1\, , \,\,
	-2\sqrt{\mu_1}+\mu_1+1	\, \bigr) , \\
	\bigl( \,\sqrt{\mu_2}-\mu_2\, , \,\,\mu_2\, , \,\,-2\sqrt{\mu_2}+\mu_2+1\, , \,\,
	\sqrt{\mu_2}-\mu_2 \, \bigr), \\
	\bigl( \, \sqrt{\mu_3}-\mu_3 \, ,\,\,-2\sqrt{\mu_3}+\mu_3+1\, ,\,\,
	\mu_3\, ,\,\,\sqrt{\mu_3}-\mu_3 \, \bigr) ,\\
	\bigl( \,-2\sqrt{\mu_4}+\mu_4+1\, ,\,\,\sqrt{\mu_4}-\mu_4\, ,\,\,
	\sqrt{\mu_4}-\mu_4\, ,\,\,\mu_4 \, \bigr), \\
\bigl( \,\mu_1\, ,\,\,\mu_2\, ,\,\, \mu_1(\mu_3{+}\mu_4)/(\mu_1{+}\mu_2)\, ,\,\, 
\mu_2(\mu_3{+}\mu_4)/(\mu_1{+}\mu_2)\, \bigr),  \\
\bigl( \,\mu_1\, ,\,\, \mu_1(\mu_2{+}\mu_4)/(\mu_1{+}\mu_3)\, ,\,\,
\mu_3 \, ,\,\, \mu_3(\mu_2{+}\mu_4)/(\mu_1{+}\mu_3) \, \bigr),
 \\
\bigl( \, \mu_2(\mu_1{+}\mu_3)/(\mu_2{+}\mu_4)\, ,\,\,
\mu_2\, ,\,\, 
\mu_4(\mu_1{+}\mu_3)/(\mu_2{+}\mu_4)\,, \,\,\mu_4 \, \bigr),
  \\
\bigl( \, 
 \mu_3(\mu_1{+}\mu_2)/(\mu_3{+}\mu_4)\, ,\,\,
 \mu_4(\mu_1{+}\mu_2)/(\mu_3{+}\mu_4)\, ,\,\,\mu_3\, ,\,\,\mu_4 
 \, \bigr). 
	\end{cases} 
	\]
	The walls of indecision are the surfaces
$\{\mu \in \Delta_3: \mu_1-\mu_4=0, \mu_1+\mu_2 \geq \sqrt{\mu_1}, \mu_1+\mu_3 \geq \sqrt{\mu_1}\}$ 
and $\{\mu \in \Delta_3: \mu_2-\mu_3=0, \mu_1+\mu_2 \geq \sqrt{\mu_2}, \mu_2+\mu_4 \geq \sqrt{\mu_2}\}$.
\end{theorem}

\begin{figure}[h]
	\centering	
	\vspace{-0.15in}
	\includegraphics[scale=0.8, trim={6.5cm 2cm 6.8cm 2cm},clip]{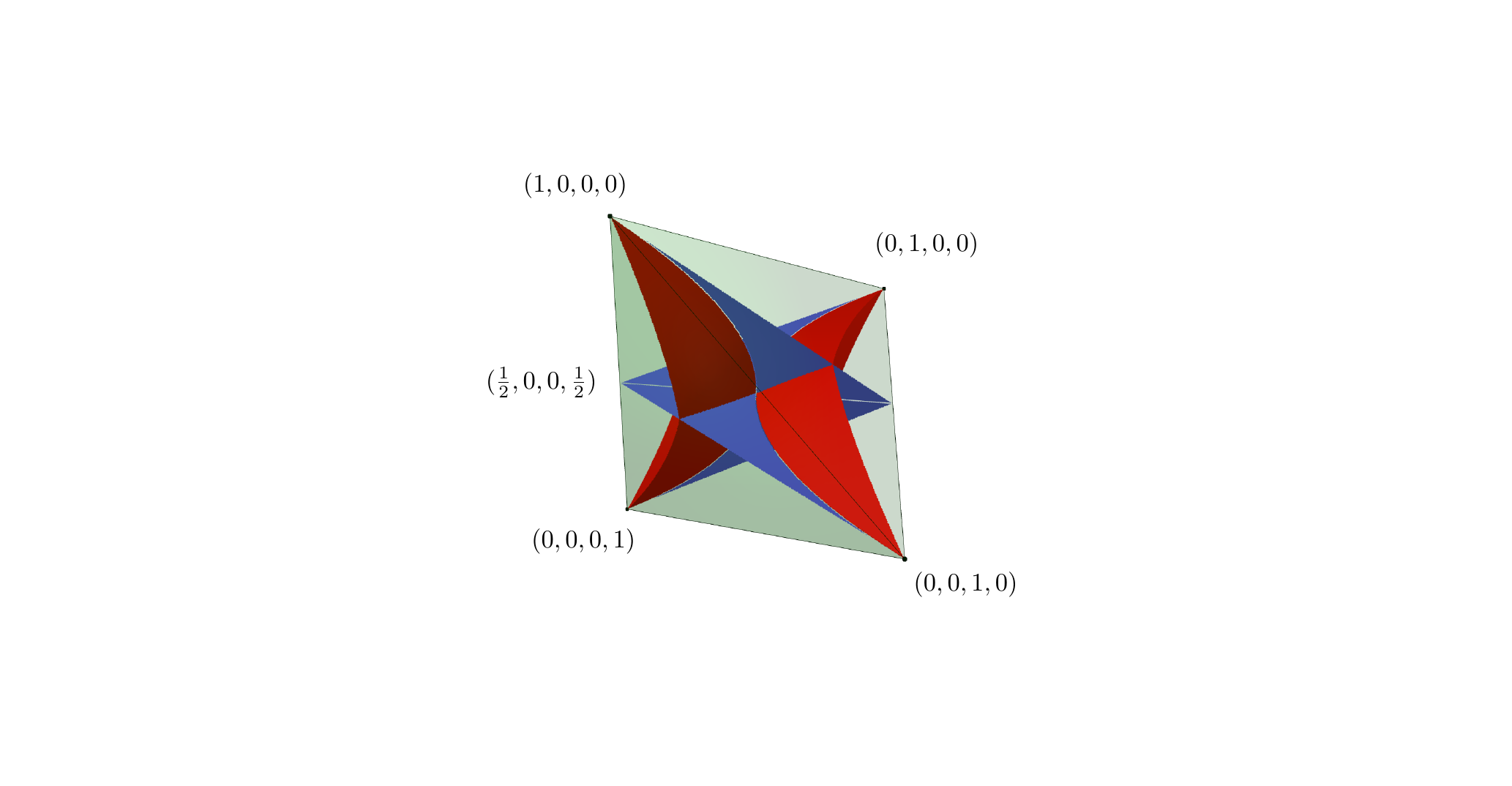}
	\qquad
	\includegraphics[scale=0.8, trim={6.8cm 2cm 7cm 2.5cm},clip]{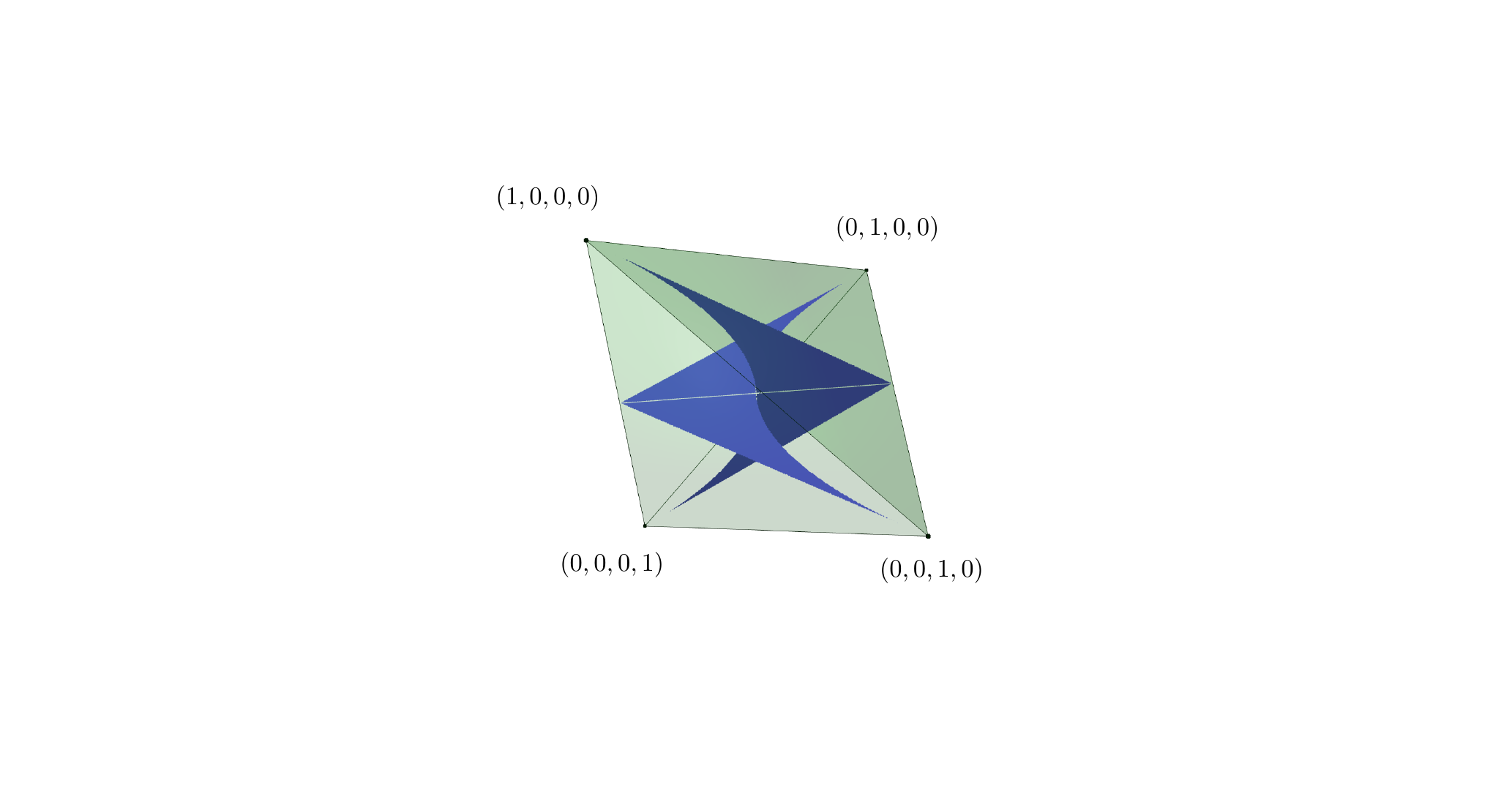}
	\vspace{-0.12in}
	\caption{The optimal value function of \Cref{thm: 3bits} subdivides the tetrahedron 
	of probability distributions $\mu$ (left). 
	The walls  of indecision are shown in blue (right).} 
	\label{fig:partition12}
\end{figure}

Theorem \ref{thm: 3bits} distinguishes eight cases. This division of $\Delta_3$ is shown in Figure~\ref{fig:partition12}.  
Each of the last four cases
breaks into two subcases, since the 
numerator in the formulas is the
absolute value of $\mu_1 \mu_4 - \mu_2 \mu_3$.
The sign of this $2 \times 2$ determinant  matters for the pieces
of our piecewise algebraic function.
The tetrahedron
$\Delta_3$ is divided into $12$ regions on which
$\mu \mapsto W_d(\mu,\mathcal{M})$ is algebraic. 

We now explain Figure~\ref{fig:partition12}. The red surface consists of eight pieces.
 Together with the blue surface, these separate the eight cases (this surface is not the model).  
Four convex regions are enclosed between the red surfaces and the 
sides they meet. 
These regions represent the first four cases in Theorem \ref{thm: 3bits}. For instance, the region containing the points $(1,0,0,0), (1/2,0,0,1/2)$ corresponds to the first case. The remaining four regions are each bounded by two red and two blue pieces, and correspond to the last four cases. Each of these four regions is further split in two by the model which we do not depict for the sake of visualization. The two 
sides are determined by the sign of the determinant $\mu_1 \mu_4 - \mu_2 \mu_3$.
The two blue shapes in the right figure form the walls of indecision. 
These specify the points $\mu\in\Delta_3$ with more than one optimal~solution.
 
 The same 2-bit model was studied in our conference paper \cite{macis}. Theorem~\ref{thm: 3bits} is a much improved representation of the results  
 in \cite[Table 2]{macis}. Our formulas can easily be translated into a
 description in terms of the parameters $(p,q)$ from (\ref{eq:model}).
The linear program we used in \eqref{eq:dual2} to
define the Wasserstein distance is dual to the one via optimal transport in
\cite[eqn (2)]{macis}. The latter primal formulation underlies the analysis in 
\cite[\S 5]{macis}. In Section \ref{sec3}, we will present a self-contained proof of Theorem \ref{thm: 3bits} after a general discussion of distance minimization for polyhedral norms.

\section{Polyhedral Norm Distance to a Variety} 
\label{sec3}

The Wasserstein metric on the simplex of probability distributions with $n$ states defines a polyhedral norm on $\mathbb{R}^m$ with $m=n-1$ as follows. 
We translate the simplex $\Delta_m$ such that its barycenter is the origin.
Next we consider a Wasserstein unit ball around the origin, denoted by $B$. This unit ball is a centrally symmetric  $m$-dimensional 
polytope $B$. It induces a norm on $\mathbb{R}^m$ by
\begin{equation*}
 \| y \|_B \,\, \, := \,\,\, {\rm min} \,\{ \,\lambda \in \mathbb{R}_{\geq 0} \,:\, y \in \lambda B \,\}. 
 \end{equation*}
In terms of the dual polytope
$$ B^* \,\,=\,\, \{ \,x \in \mathbb{R}^{m} \,:\, \sup_{z\in B} \langle x, z \rangle \leq 1 \,\},$$ 
the polyhedral norm can be rewritten as
\begin{equation*}
\| y \|_B \,\, \, = \,\,\, {\rm min} \,\{  \,\lambda\in\mathbb{R}_{\geq 0} \,:\,
\sup_{x\in B^*} \langle x, y \rangle \leq \lambda \,\} \,\,\, = \,\,\,
\max_{x\in B^*} \, \langle x,y \rangle.
\end{equation*}
Note that $(B^*)^* = B$. 
The dual of the unit ball equals
$$ B^* \,\, = \,\,
 P_d \,\, = \,\, \bigl\{\,x \in \mathbb{R}^n /\mathbb{R}{\bf 1} \,\,:\,\,\, |x_i - x_j| \,\leq\, d_{ij} \,\,\,\,
\text{for all} \,\, \,1\leq i < j \leq n \,\bigr\}.  $$
This is the Lipschitz polytope in \eqref{eq:polytope}, and
the unit ball $B  = P_d^*$ is its dual. 
This means that the Wasserstein unit ball $B$ is the convex hull of $n(n-1)$ vectors
that lie on a hyperplane in $\mathbb{R}^n$:
$$ B \,\, = \,\, P_d^* \,\,=\,\,\,
{\rm conv}\,
\biggl\{\,\frac{1}{d_{ij}} (e_i - e_j)\,\,:\,  \,\, \,1\leq i < j \leq n \,\biggr\}.  $$
In the case $m=n-1=2$,
two Wasserstein balls for different metrics $d$ 
were shown in Figure~\ref{fig:wassersteinball}. 

\begin{example}\label{ex: discrete_metric_3}
Fix $m=n-1=3$  and let $d$ be the $2$-bit Hamming metric in (\ref{eq:metric}).
We work in the linear space $L$ that is defined by $x_1+x_2+x_3+x_4=0$.
The Lipschitz polytope is the octahedron
\begin{align*} P_d & = B^* =
\, \, \bigl\{ \,(x_1,x_2,x_3,x_4) \in L\,:\,
|x_1-x_2| \leq 1,\,
|x_1-x_3| \leq 1,\,
|x_2-x_4| \leq 1,\,
|x_3-x_4| \leq 1 \,\bigr\} \\ & =
{\rm conv} \bigl\{ (1,0,0,-1), (1,0,0,-1), (\tfrac{1}{2}, -\tfrac{1}{2}, -\tfrac{1}{2},\tfrac{1}{2}), (-\tfrac{1}{2}, \tfrac{1}{2}, \tfrac{1}{2},-\tfrac{1}{2}), (0, 1, -1,0), (0, -1, 1,0)\bigr\}.
\end{align*}
%
The Wasserstein unit ball is the cube
$$ \begin{matrix} B\, =\, P_d^* & \!\! = \! &
\bigl\{ \,(y_1,y_2 , y_3,y_4) \in L\,:\,
|y_1 - y_4| \leq 1,\,
|y_2-y_3| \leq 1,\,
|y_2 + y_3| \leq 1  \bigr\} \\  & \!\! = \! & \hspace{-15mm}
{\rm conv} \bigl\{
(1,-1,0,0),(1,0,-1,0),(0,1,0,-1),(0,0,1,-1) \\
& & (-1,1,0,0),(-1,0,1,0), (0,-1,0,1),(0,0,-1,1) \bigr\}  .
\end{matrix} $$
\end{example}

Returning to the general case, suppose that $\mathcal{M}$ is a smooth compact algebraic variety in $\mathbb{R}^m$.
For any point $u \in \mathbb{R}^m$, we are interested in its distance to the variety under our polyhedral norm:
\begin{equation*}
 D_B(u,\mathcal{M}) \,\,\, := \,\,\,
{\rm min} \bigl\{ \,\| u-v \|_B \,:\, v \in  \mathcal{M} \,\bigr\} \,\,\, = \,\,\,
{\rm min} \bigl\{ \,\lambda \in \mathbb{R}_{\geq 0} \,:\,
(u + \lambda B ) \,\cap \, \mathcal{M} \,\not= \, \emptyset \,\bigr\}. 
\end{equation*}

We will now embark on understanding the geometry of this optimization problem.

\begin{proposition}\label{prop: unique}
If the model $\mathcal{M}$ and the point $u$ are in general position relative to the unit ball $B$ then 
there is a unique optimal point $v \in \mathcal{M}$ for which  $D_B(u,\mathcal{M})  = \|u-v\|_B = \lambda$ holds.
The point $\frac{1}{\lambda}(v-u)$ is in the relative interior of a unique face $F$ of the polytope $B$; we say that
$v$ has {\em type}~$F$.
\end{proposition}

The general position hypothesis is understood as follows.
 The rotation group and the translation group act on $\mathbb{R}^m$.
These two algebraic groups have  Zariski dense subsets such that the hypothesis holds after
applying group elements from those two subsets to $\mathcal{M}$ and $u$ respectively.

\begin{proof} We have $\lambda = D_B(u,\mathcal{M})$, so 
$\frac{1}{\lambda}(v-u)$ lies in the boundary of the unit ball $B$.
The polytope $B$ is the disjoint union of the relative interior
of its faces. Hence there exists a unique face $F$ that has
$\frac{1}{\lambda}(v-u)$ in its relative interior.
Let $L_F$ be the linear subspace of $\mathbb{R}^{m}$ 
that consists of linear combinations of vectors in $F$.
By hypothesis, the resulting affine subspace $u + L_F$ intersects 
the variety $\mathcal{M}$ transversally, and $v$ is a 
general smooth point in that intersection. Moreover,
$v$ is a minimum of the restriction to the variety  $(u+L_F) \cap \mathcal{M}$
of a linear function on $u+L_F$. Our hypothesis ensures that
the  linear function is generic relative to the variety,
which in turn is smooth and compact. The number of
critical points is finite. This guarantees that the linear
function attains its minimum at a unique point in the variety, namely at $v$.
\end{proof}

\smallskip

Our geometric discussion becomes very concrete in the
Wasserstein case. The data point is $u = \mu$ and the optimal point is $v = \nu^*$.
The type of $v$ is a face $F$ of the unit ball $B = P_d^*$.
Fix the face $F$. This allows for the following algebraic characterization of optimality.
Let $\mathcal{F}$ be the set of all index pairs $(i,j)$ 
such that the point $\frac{1}{d_{ij}}(e_i-e_j)$ is a vertex and it lies in $F$.
Let $\ell_F$ be any linear functional on $\mathbb{R}^{m}$
that attains its maximum over $B$ at $F$.
We work in the linear~space
\begin{equation}
\label{eq:elleff}
 L_F \,\,\, = \,\, \,
\left\{ \, \sum_{(i,j) \in \mathcal{F}} \!\! \lambda_{ij} (e_i - e_j) \,\,:
\,\, \lambda_{ij} \in \mathbb{R} \,\right\}. 
\end{equation}

The point $\nu^*$  on $\mathcal{M}$ that is closest to $\mu$ is the
solution of the following optimization problem:
\begin{equation}
\label{eq:generalopt}
 \hbox{{\rm Minimize} $\,\ell_F = \ell_F(\nu)\,$ subject to $\,\nu \in (\mu + L_F) \cap \mathcal{M}$.} 
\end{equation} 
This is a polynomial optimization problem in the linear subspace $L_F$ of $\mathbb{R}^{m}$.
With the notation in (\ref{eq:elleff}), the
decision variables are $\lambda_{ij}$ for  $(i,j) \in \mathcal{F}$.
The algebraic complexity of this  problem 
will be studied in Section \ref{sec5}. In Section \ref{sec4}, we
focus on the combinatorial complexity.
 The unit ball $B$ has very many faces, and
our desire is to control that combinatorial explosion.
For the remainder of this section, we return to the three-dimensional case seen
in Section \ref{sec2}, and we present a proof of Theorem~\ref{thm: 3bits}
that uses  the set-up above.
Theorem \ref{thm: Hardy-Weinberg} is analogous and its proof will be omitted.
\begin{figure}[h]
	\centering
	\includegraphics[scale=0.65]{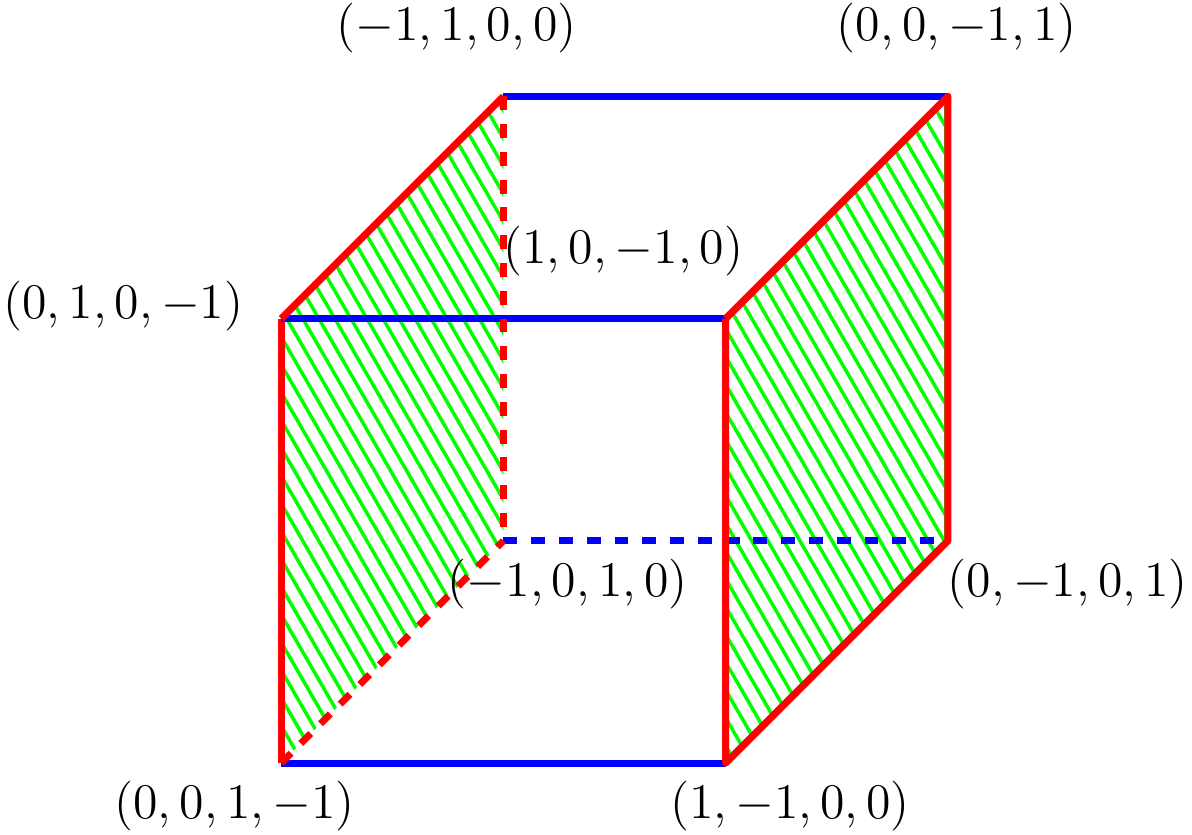}
	\caption{Subdivision of the faces of the Wasserstein ball as in the proof of \Cref{thm: 3bits}.}
	\label{fig:proof_2.2}	
\end{figure}

\begin{proof}[Proof of Theorem~\ref{thm: 3bits}]
The Wasserstein unit ball is the cube $B$ in \Cref{ex: discrete_metric_3}. We must solve \eqref{eq:generalopt} for every face $F$ of $B$. There are various symmetries we can employ to simplify the proof. First, since $B$ is centrally symmetric, we study only one among a face
$F$ and its negative $-F$. Since $L_F=L_{-F}$, minima in \eqref{eq:generalopt} for
$F$ turn into maxima for $-F$, and vice versa. Second, consider the dihedral group $D_4
$ of order $8$ that is generated by the involutions
$ (14) $ and $(12)(34)$ in the symmetric group on $\{1,2,3,4\}$. This 
acts on the tetrahedron $\Delta_{3}$, on the cube $B$, and on the model $\mathcal{M}$,
by permuting coordinates in $\mathbb{R}^4$. The action
respects scalar products: $\langle c,x \rangle=\langle g\cdot c,g\cdot x \rangle$ for every $g\in D_4$. Therefore, $g\cdot F$ is a face of $B$ for every face $F$
and every $g \in D_4$, and the problem \eqref{eq:generalopt} is symmetric under $D_4$.
The solution function satisfies
 $\,\nu^*(g\cdot \mu)=g \cdot \nu^*(\mu)\,$ for all $g \in D_4$.
 
For each vertex, edge or $2$-face, one per symmetry class, we introduce Lagrange multipliers to compute the
critical points of (\ref{eq:generalopt}). In each case, there
are at most two critical points, since the
polar degrees are $\delta=(2,2,2)$;
see~$k=2$ in Table~\ref{tab:dims}.
We now undertake a case-by-case analysis:
\begin{itemize}
		\item $\dim(F)=2$:	The green facets in
		 \Cref{fig:proof_2.2} give two orbits. For the first facet, Lagrange multipliers reveal a critical point $\nu^*=(1/4,1/4,1/4,1/4)$. However, the associated constrained Hessian is indefinite, and hence $\nu^*$ is not a local minimum. 	 The second facet
 has no critical points in $\Delta_3$. Hence there is never any optimal solution whose type is a facet.
 		\item $\dim(F)=1$: We have two orbits of edges, marked in red (bounding the green facets) and blue in \Cref{fig:proof_2.2}. 
		Representatives are
		 $E_1=\conv\{(-1,1,0,0),(-1,0,1,0)\}$ and $E_2=\conv\{(1,-1,0,0),(0,0,1,-1)\}$. 
		 For the first, we have $L_{E_1}=\{x_4=0, x_1+x_2+x_3=0\}$ and $\ell_{E_1}=-x_1+x_4$. The associated Lagrangian system has two solutions one of which is contained in $\Delta_3$, namely $\nu^*=(-2\sqrt{\mu_4}+\mu_4+1,\sqrt{\mu_4}-\mu_4,\sqrt{\mu_4}-\mu_4, \mu_4)$. The constrained Hessian reveals that $\nu^*$ is a local minimum. It remains to determine the constraints of the region on which $\nu^*$ lies in the interior of $E_1$. 
They can be obtained from the inequalities defining the  $2$-dimensional~cone
\[ C_{E_1}\,\,:= \,\,\{\,\lambda_{12}(e_2-e_1)+\lambda_{13}(e_3-e_1)\,\,:\,\,
 \lambda_{12}, \lambda_{13} \in\mathbb{R}_{\geq 0}\,\}. \]
		Then $\nu^*\in \mu + C_{E_1}$ if and only if $\nu^*_2-\mu_2\geq 0$ and $\nu^*_3-\mu_3\geq 0$, that is $\sqrt{\mu_4}-\mu_4-\mu_2\geq 0$ and $\sqrt{\mu_4}-\mu_4-\mu_3\geq 0$. As $\ell_{E_1}=-x_1+x_4$, the corresponding optimal 		 Wasserstein distance is 
$$ \quad W(\mu,\nu^*)\,\,\,=\,\,\,
\ell_{E_1}(\nu^* - \mu)
\,\,\, = \,\,\,
2\sqrt{\mu_4}+\mu_1-\mu_4-1\,\,=\,\,2\sqrt{\mu_4}(1-\sqrt{\mu_4})-\mu_2-\mu_3. $$
		The optimization problem associated to $E_2$ does not have critical points.
\item $\dim(F)=0$: The eight vertices of $B$ form one orbit. We consider $v=(1,-1,0,0)$, with associated
		zero-dimensional variety	$(\mu + L_v)\cap \mathcal{M} $. This consists of a unique point
	 $\nu^*=(\frac{\mu_3(\mu_1+\mu_2)}{\mu_3+\mu_4}, \frac{\mu_4(\mu_1+\mu_2)}{\mu_3+\mu_4},\mu_3,\mu_4)$. Depending on $\mu$, this point can lie either on the ray through $\mu+v$, denoted $\mu+C_v$, or on the ray through $\mu-v$. We have $\nu^*\in  (\mu+C_{v})\cap\mathcal{M}$ if and only if $\nu^*_1-\mu_1\geq0$, that is $\frac{\mu_2\mu_3-\mu_1\mu_4}{\mu_3+\mu_4}\geq 0$.
		In this case we choose $\ell_{v}=-x_2-x_3$, and we~obtain 
$$				
W(\mu,\nu^*)\,\,\,=\,\,\, \ell_{v}(\nu^* - \mu) \,\,\, = \,\,\,		
		-\frac{\mu_4(\mu_1+\mu_2)}{\mu_3+\mu_4}-\mu_3+\mu_2+\mu_3
	\,	\,\,=\,\,\, \frac{\mu_2\mu_3-\mu_1\mu_4}{\mu_3+\mu_4}. $$
	\end{itemize}
We act with the dihedral group $D_4$ on the two local minima we found. 
This yields the eight expressions for $\nu^*$  shown in  Theorem~\ref{thm: 3bits}.
It remains to decide which point $\nu^*$ is the global minimum. This is done by pairwise comparison
of the eight expressions for the Wasserstein distance $W_d(\mu,\nu^*)$.
 We omit this last step, since it consists of elementary algebraic manipulation.
\end{proof}

\section{Lipschitz polytopes}                    \label{sec4}

The combinatorial complexity of our problem is governed by the facial structure of the
Wasserstein ball given by a finite metric space $([n],d)$. We now focus on the polar dual
of that ball, which is the Lipschitz polytope $\,P_d $. This lives in
$\,\mathbb{R}^n/\mathbb{R}{\bf 1} \simeq \mathbb{R}^{n-1}$,
and is defined in~(\ref{eq:polytope}).

This object appears in the literature in several guises. 
See e.g.~\cite{GP} for a study that emphasizes generic distances $d_{ij}$.
  We consider specific metrics that are relevant for the independence~model: 
   \begin{itemize}
\item The discrete metric on any finite set $[n]$ where $\,d_{ij} = 1$ for  distinct $i,j$. 
	\item The $L_0$-metric on $[m_1]\times\dots\times[m_k]$ where $\,d_{ij} = \# \{l: i_l\neq j_l \}$. 
	\item The $L_1$-metric on $[m_1]\times\dots\times[m_k]$ where $\,d_{ij} =\sum_{l=1}^k | i_l-j_l|$.
\end{itemize}
For the last two metrics we have $n=m_1 \cdots m_k$.
To compute the Wasserstein distance in each case,  we need to
 describe the Lipschitz polytope $P_d$ as explicitly as possible.
 All three metrics above can be interpreted as \emph{graph metrics}. 
 This means that there exists an undirected simple graph 
 $G$ with vertex set $[n]$  such that $d_{ij}$ is the length of the shortest path from $i$ to $j$ in $G$. 
 Wasserstein balls associated to graphs in this way are 
 studied  in \cite{DDM} under the name \emph{symmetric edge polytopes}.

  For the discrete metric on $[n]$, the graph is the complete graph $K_n$. In the case of the $L_0$-metric on $[m_1]\times\dots\times[m_k]$, we have the Cartesian product of complete graphs $K_{m_1}\times\dots\times K_{m_k}$. 
  In the last case, the corresponding graph is the Cartesian product of paths of length $m_1,\dots,m_k$. 
  The facets of the Lipschitz polytope $P_d$ arising from a graph $G$ correspond to the edges of $G$. We have
  	\begin{equation}
P_{d}\,\,=\,\,\{\,x\in \mathbb{R}^n/\mathbb{R}{\bf 1}\,:\, |x_i-x_j|\leq 1 \, \text{ for every edge } (i,j) \,\,{\rm of}\,\, G \,\}. 
	\label{eq: Lipschitz_reduced}
	\end{equation}
This representation of $P_d$ is a consequence of the triangle inequality.
Vertices of 
$P_d$ are 
precisely those points for which at least $\dim(P_{d})$
inequalities are sharp. More generally, we are interested in higher-dimensional faces of $P_d$. The number of $i$-dimensional faces of $P_d$ is denoted by $f_i = f_i(P_d)$,
and we write  $f= (f_0,f_1,\ldots,f_{n-2})$ for the {\em f-vector}. 
Since $P_d$ is $(n-1)$-dimensional,
we have $f_{n-1}(P_d)=1$, and we omit this number. In general, it is difficult to compute the $f$-vector.
 
If $d$ is the discrete metric on $[n]$, then
we have the following description of the faces.
The corresponding Lipschitz polytope
$P_d$ is a zonotope, namely it is  the Minkowski sum of $n$ general segments
in $(n-1)$-space. For $n=4$ this is the rhombic dodecahedron \cite[Figure 4]{JK}.
Its dual, the Wasserstein ball for the discrete metric on $[n]$,
is  the \emph{root polytope} of Lie type A; cf.~\cite{JK, Ngoc}.

\begin{lemma} \label{lem:rhombic}
Let $d$ be the discrete metric on $[n]$. The vertices of $P_d$ are the binary vectors
$\sum_{i \in I} e_i$	where $I$ runs over elements of the power set $\, 2^{[n]}\backslash \{\emptyset,[n]\}$. Furthermore,
	   a subset $\,S$ of $2^{[n]}\backslash \{\emptyset,[n]\}$ 
indexes the vertices of  a face of $P_d$ if and only if 
$\,S=\{I: L\subseteq I \subseteq U\}$ for some 
$L,U\in 2^{[n]}\backslash \{\emptyset,[n]\}$.
\end{lemma}

\begin{proof}
Clearly, $e_I=\sum_{i\in I} e_i$ lies in $P_d$. We observe that $(e_I)_i-(e_I)_j=1$ if and only if $i\in I$ and $j\notin I$. The corresponding linear forms $x_i - x_j$
for $i \in I$ and $j \not\in I$ span an $(n-1)$-dimensional space.
This means that $e_I$ is a vertex of $P_d$.
Conversely, there are no vertices other than the $e_I$ since
 $v_i-v_j=1$ implies $v_i=1$ and $v_j=0$ for $v\in \mathbb{R}^n/\mathbb{R}{\bf 1}$.
   For the second statement, consider any linear functional $\ell$ on $P_d$.
   We have $\ell = \sum_{i=1}^n a_i x_i$ where $\sum_{i=1}^n a_i = 0$.
   Set $   L = \{i : a_i > 0\}$ and $U = \{i: a_i \geq 0\}$.
   Then $\ell$ is maximized over $P_d$ at the convex hull
   of $\{ e_I \, : \, L \subseteq I \subseteq U \}$, so this is a face.
   Every face is the set of maximizers of a linear functional   on $P_d$.
   This proves the claim.
\end{proof}   

From this description of  $P_d$ we can read off the number of faces in each dimension.

\begin{corollary}\cite[Proposition 4.3]{CM}
Let $d$ be the discrete metric on $[n]$. Then	
$$ f_i(P_d)\,=\, f_{n-i-2}(P_d^*)\,\,=\,\,\binom{n}{i}(2^{n-i}-2) \qquad
\hbox{ for  $\,i=0,\dots,n-2$.} $$
\end{corollary}

\begin{proof}
The face indexed by $(L,U)$ in the proof of Lemma
\ref{lem:rhombic} has dimension $|U| - |L| $.
Hence $f_i$ is the  number of chains
$\,\emptyset \subsetneq L \subseteq U \subsetneq [n]\,$
with $|U| - |L| = i$. This is the given number.
\end{proof}

\begin{example}[$n=4$]
We consider the discrete metric on $[4] = \{1,2,3,4\}$. 
The $3$-dimensional Lipschitz polytope $P_d$ is the rhombic dodecahedron with $f$-vector $(14,24,12)$. Its dual $P_d^*$ is the Wasserstein
ball with $f$-vector $(12,24,14)$. The normal fan of $P_d$, which is
the fan over $P_d^*$, is a central arrangement of four general planes in a $3$-dimensional space. This has $14$ regions.
\end{example}

\begin{corollary}  Up to a factor of 2, the Wasserstein distance between probability
	distributions on $[n]$ is the restriction of the $L_1$-distance on $\mathbb{R}^n$. In symbols $W_d =\frac{1}{2}
	\|\mu-\nu\|_{L_1}$ for $\mu,\nu\in\Delta_{n-1}$.
\end{corollary}

\begin{proof} Up to a factor of $2$, which we ignore,
$P_d$ is the image of the cube $[-1,1]^n$ under the map
$\mathbb{R}^n \rightarrow \mathbb{R}^n / \mathbb{R} \mathbf{1}$. 
 Hence its dual,
which is the $L_1$-ball or cross polytope, intersects the hyperplane
$\mathbf{1}^\perp$ in the Wasserstein ball $P_d^*$. This means that the
$L_1$-metric agrees with the Wasserstein metric on any translate of $\mathbf{1}^\perp$.
More explicitly, we compute $W_d(\mu,\nu)$ with the formula \eqref{eq:dual2}. 
This yields
$$  W_d(\mu,\nu)\,\,=\,\,\max_{x\in P_d}\langle\mu-\nu,x\rangle
\,\,=\,\,\langle\mu-\nu,\text{sign}(\mu-\nu)\rangle\,\,=\,\,\sum_{i=1}^n |\mu_i-\nu_i|. $$
Here we identify the linear functionals given by the vertices of $2 P_d$ 
 with elements in $\{-1,1\}^n$.
\end{proof}

\begin{example}
The $L_1$-ball for $n=3$ is an octahedron. The restriction of this octahedron to the triangle $\Delta_2$
is the hexagon on the left of \Cref{fig:wassersteinball}. 
\end{example}

We next examine the Lipschitz polytope $P_d$
for metrics associated to graphs $G$ other than $K_n$. 
The inequality representation was given in 
(\ref{eq: Lipschitz_reduced}). However, describing all faces,
or even just the vertex set $V(P_d)$,
is now more difficult than in Lemma~\ref{lem:rhombic}.
The Wasserstein ball $P_d^*$ is
the convex hull of the subset of vertices $e_i-e_j$ of
the  root polytope of type A that are indexed by edges of~$G$.
The following result for bipartite graphs $G$ is due to \cite[Lemma 4.5]{DDM}. A related characterization for weighted graphs was obtained in \cite[Theorem 2, \S 3.1]{montrucchio2019}.

\begin{proposition}\label{lem: bipartite}
	Let $d$ be a graph metric where $G$ is bipartite.
	The set of vertices of $P_d$ equals
\begin{equation}
	V(P_{d})\,\,=\,\,\{\,x\in \mathbb{Z}^n/\mathbb{Z}{\bf 1}\,:\, |x_i-x_j| = 1 
	\, \text{ for every edge } (i,j) \,\,{\rm of}\,\, G \,\}. 
	\label{eq: vertices_bipartite}
	\end{equation}
\end{proposition}	

\Cref{lem: bipartite} covers the case of the Lipschitz polytope for the $L_1$-norm on a product of finite sets. In particular, we obtain a vertex description for the Lipschitz polytope of the graph of the $k$-cube.
This covers the $L_0$-metric which is equal to the $L_1$-metric on the states of the $k$-bit models. 
This metric is the \emph{Hamming distance} on a cube. In \Cref{ex: discrete_metric_3}, we described this for the $2$-bit model, for which the Lipschitz polytope is an octahedron, and its dual is a cube.

 It is not easy to compute the cardinality of \eqref{eq: vertices_bipartite}. In graph theory, this corresponds to counting graph homomorphisms from the $k$-cube to the infinite path with a fixed point. 
 \cite{Gal} observed that there is a bijection between $V(P_d)$ and the proper $3$-colorings of $k$-cube with a vertex with fixed color. For $k=2,3,4,5,6$, the corresponding number equals
$\, 6,\, 38,\,990,\, 395094,\,33433683534$. This was computed with the graph coloring code in
{\tt SageMath}. We refer to \cite{Gal} for  asymptotics.

It follows from results in  \cite{JK} that the Wasserstein ball for the discrete metric on $[n]$ has
the most vertices for any metric on $[n]$.
We next discuss the Wasserstein ball with the fewest vertices.

\begin{example} \label{lem: L1 has a cube}
Let $d$ be the $L_1$-metric on $[n]$, i.e.~the graph metric of the $n$-path.
Then $\,P_d \, = \, \{ |x_i - x_{i+1}| \leq 1 \,: \,i=1,2,\ldots,n-1 \}\,$ is
	combinatorially an $(n-1)$-cube, and 	$P_d$ is a cross polytope.
	This has the minimum number of vertices for any centrally symmetric $(n-1)$-polytope:
	\[ \qquad f_i(P_d)\,=\, f_{n-i-2}(P_d^*)\,=\,2^{n-i-1}\binom{n-1}{i} \qquad
	\hbox{for}\,\,\, i=0,1,\dots,n-2.	\]
\end{example}

We conclude this section with four independence models that serve as examples 
for our case studies in the next sections. The tuple $(({m_1})_{d_1},\dots , ({m_k})_{d_k})$
denotes the independence model with  $n = \prod_{i=1}^k \binom{m_i+d_i-1}{d_i}$ states
where the $i$th entry $(m_i)_{d_i}$ refers to a multinomial distribution with $m_i$ possible outcomes and $d_i$ trials.
This can be interpreted as an unordered set of $d_i$ identically distributed random variables on $[m_i]=\{1,2,...,m_i\}$. The subscript $d_i$ is omitted if $d_i = 1$.

For example, $(2_2,2)$ denotes the independence model for
three binary random variables where the first two are identically distributed.
We list the $n=6$ states in the order $00, 10, 20, 01, 11, 21$.
These are the vertices of the associated graph $G$, which is the product of a $3$-chain and a $2$-chain.
This  model $\mathcal{M}$ is the image of the map from the square $[0,1]^2$
into the simplex $\Delta_5$ given by
\begin{equation}
\label{eq:zweizweizwei}
(p,q) \,\mapsto \, \bigl(\, p^2q,\,2p(1-p)q, \,(1-p)^2q,\,p^2(1-q),\,2p(1-p)(1-q),\,(1-p)^2(1-q)\, \bigr).
\end{equation}

\begin{example}\label{ex: running_example}
Our four models are:
 the $3$-bit model $(2,2,2)$ with the $L_0$-metric on $[2]^3$;
 the model $(3,3)$ for two ternary variables with the $L_1$-metric on $[3]^2$;
the model $(2_6)$ for six   identically distributed binary variables with the discrete metric on $[7]$; 
the model  $(2_2,2)$ in (\ref{eq:zweizweizwei}) with the $L_1$-metric on $[3]\times[2]$. In \Cref{table_experiments_combinatorics}, we report the $f$-vectors of the corresponding Wasserstein balls.

\vspace{5pt}
\begin{table}[h!]
	\begin{center}
		\vspace{-0.14in}
		\renewcommand*{\arraystretch}{1.2}
		\begin{tabular}{ | l | l | l |  l |  l |}
			\hline
			$\mathcal{M}$ & $n$ & $\dim(\mathcal{M})$ & Metric $d$ & $f$-vector of the 
			$(n{-}1)$-polytope $P_d^*$ \\ \hline
			$(2,2,2)$ & 8 & 3 & $L_0=L_1$ & $(24, 192, 652, 1062, 848, 306, 38) $ \\ 
			$(3,3)$ & 9 & 4 & $L_1$ & $(24, 216, 960, 2298, 3048, 2172, 736, 82)$ \\ 
			$(2_6) $ & 7 & 1 & discrete&  $(42, 210, 490, 630, 434, 126)$ \\ 
			$(2_2,2)$ & 6 & 2 & $L_1$ & $(14,60,102,72,18)$  \\\hline
		\end{tabular}
		\vspace{-0.14in}
	\end{center}
	\caption{$f$-vectors of the Wasserstein balls for the four models in \Cref{ex: running_example}.}
	\label{table_experiments_combinatorics}
\end{table}	
\end{example}

\section{Polar Degrees of Independence Models}       \label{sec5}

In this section, we examine the problem
(\ref{eq:generalopt}) for fixed type $F$ from the perspective of algebraic geometry.
Given a compact smooth algebraic variety $\mathcal{M}$
in $\mathbb{R}^m$, we consider a linear functional $\ell$
and  an affine-linear space $L$ of dimension $r$ in $\mathbb{R}^m$.
It is assumed that the pair $(\ell,L)$ is in general position 
relative to $\mathcal{M}$.  Our aim is to study the following optimization problem:
\begin{equation}
\label{eq:generalopt2}
\hbox{Minimize the linear functional} 
\,\,\,\ell\, \,\,  \hbox{over the intersection}
\,\, L \,\cap \, \mathcal{M} \,\, {\rm in} \,\, \mathbb{R}^m.
\end{equation}
This is a constrained optimization problem. We
write the critical equations as a system of polynomial equations.
Its unknowns are the $m$ coordinates of $\mathbb{R}^m$
plus various Lagrange multipliers. The genericity assumption allows us
to attach an algebraic degree to this optimization problem.
That degree is the number of complex solutions
to the critical equations. Assuming $(\ell,L)$ to be generic,
this number does not depend on the choice of $(\ell, L)$ but just on
the dimension $r$ of $L$.
The following result furnishes a  recipe for assessing
the algebraic complexity of our problem.

\begin{theorem} \label{thm:algdeg}
	The algebraic degree of the problem
	(\ref{eq:generalopt2}) is the polar degree $\delta_r$ of $\mathcal{M}$.
\end{theorem}

We begin by explaining this statement. First of all,
we already tacitly replaced $\mathcal{M}$ by
its closure in complex projective space $\mathbb{P}^m$,
and we are assuming that this projective variety is smooth.
Let $(\mathbb{P}^m)^\vee$ denote the dual projective space
whose points are the hyperplanes $h$ in $\mathbb{P}^m$.
The {\em conormal variety} of the model $\mathcal{M}$ is the following 
subvariety in the product of two projective spaces:
$$ CV(\mathcal{M}) \quad = \quad
\bigl\{ \,(x,h) \in \mathbb{P}^m \times (\mathbb{P}^m)^\vee \,\,:\,\,
\hbox{the point $ x$ lies in $ \mathcal{M} $ and $h$ is tangent to $\mathcal{M}$ at $x$}\, \bigr\}. $$
The importance of the conormal variety for optimization has been explained in several sources,
including \cite{DHOST, NRS, RS}. The projection of $CV(\mathcal{M})$ onto 
the second factor $(\mathbb{P}^m)^\vee $ is the {\em dual variety} $\mathcal{M}^*$,
which parametrizes hyperplanes that are tangent to $\mathcal{M}$.
It is known that $CV(\mathcal{M}^*) = CV(\mathcal{M})$ and
that this conormal variety always has dimension $m-1$;
see \cite[Proposition~2.4 and Theorem~2.6]{RS}.
The dual variety already appeared in \cite[\S 4]{macis},
but here we need a more general approach.

Let $[CV(\mathcal{M})]$ denote the class of the conormal variety in
the  cohomology of $\mathbb{P}^m \times (\mathbb{P}^m)^\vee $.
This  cohomology ring is $\mathbb{Z}[s,t]/\langle s^{m+1}, t^{m+1} \rangle$, 
and hence the class
$[CV(\mathcal{M})]$ is a homogeneous polynomial of
degree $m+1$ in two unknowns $s$ and $t$. We can write this binary form as follows:
\begin{equation}
\label{eq:cohomclass}
[CV(\mathcal{M})] \quad = \quad \sum_{r=1}^{m} \,\delta_{r-1} \cdot s^r t^{m+1-r} . 
\end{equation}
The coefficients $\delta_0,\delta_1,\delta_2, \ldots$ are  the
{\em polar degrees} of the model~$\mathcal{M}$. 
Some of these are zero.
Namely, the sum in
(\ref{eq:cohomclass}) ranges from $r_1 $ to $r_2 $, where ${\rm dim}(\mathcal{M}) = m-r_1$
and ${\rm dim}(\mathcal{M}^*) = r_2$. The first and last non-zero coefficients are
$\delta_{r_1-1} = {\rm degree}(\mathcal{M})$ and $\delta_{r_2-1} = {\rm degree}(\mathcal{M}^*)$
respectively.

\begin{proof}[Proof of Theorem \ref{thm:algdeg}]
	It is known that $\delta_{r-1}$ equals the number of points
	in $(L_r \times L'_{m+1-r}) \cap CV(\mathcal{M})$
	where $L_r \subset \mathbb{P}^m$ is a general
	linear space of dimension $r$ and 
	$L'_{m+1-r} \subset (\mathbb{P}^m)^\vee$
	is a general linear space of dimension $m{+}1{-}r$; see e.g.~\cite[\S 5]{DHOST}.
	We now identify $L_r$ with the linear space $L$ in 
	(\ref{eq:generalopt2}). The intersection
	$(L_r \times (\mathbb{P}^m)^\vee) \,\cap \,CV(\mathcal{M})$
	is a smooth variety of dimension $r-1$ by Bertini's Theorem.
	In (\ref{eq:generalopt2}), we  optimize a general linear functional
	over its projection into the first factor $\mathbb{P}^m$. The dual variety to that projection
	lives in $(\mathbb{P}^m)^\vee$, and 
	the desired algebraic degree is the degree of the dual 
	variety. This is obtained geometrically by intersecting with $L'_{m+1-r}$.
	\end{proof}

The independence models treated in this article  are known in algebraic geometry as 
Segre-Veronese varieties.  The study of characteristic classes  for these families 
is a classical subject in algebraic geometry. The explicit computation of these polar degrees 
was carried out only recently,  in the doctoral dissertation \cite{Luca}.  
The result is described in Theorem \ref{prop:luca} below.

Let $\mathcal{M}$ be the model denoted $(({m_1})_{d_1},\dots , ({m_k})_{d_k})$  
in Section \ref{sec4}. The corresponding  Segre-Veronese variety is the embedding of
$\mathbb{P}^{m_1-1}\times \dots \times \mathbb{P}^{m_k-1}$ in the space 
of partially symmetric tensors, 
$\,\mathbb{P}({\rm Sym}_{d_1}\mathbb{R}^{m_1}\otimes \dots \otimes {\rm Sym}_{d_k} \mathbb{R}^{m_k} )$.
That projective space equals $\mathbb{P}^{n-1}$ where  $n = \prod_{i=1}^k \binom{m_i+d_i-1}{d_i}$.
We identify its real nonnegative points with the simplex $\Delta_{n-1}$.
The independence model $\mathcal{M}$ consists of the rank one tensors.
Its dimension is denoted  $\,{\bf m}:=(m_1-1)+ \dots + (m_k-1)$. 
 The following  formula for the polar degrees of the Segre-Veronese variety
 $\mathcal{M}$ appears in \cite[Chapter~5]{Luca}.

\begin{theorem} \label{prop:luca}
 For each integer $r$ with $n-1-{\rm{dim}} (\mathcal{M}) \leq r \leq \dim(\mathcal{M^*})$,
  the polar degree equals
\begin{equation}\label{LucaForm}
\delta_{r-1}(\mathcal{M})
\,\,\, =\,\,\sum_{s=0}^{{\bf m}-n+1+r}(-1)^s\binom{{\bf m}-s+1}{n-r}({\bf m}-s)!\left( \sum_{i_1+\dots+ i_k=s} \prod_{l=1}^k \frac{\binom{m_l}{i_l}d_l^{m_l-1-i_l}}{(m_l-1-i_l)!} \right).
\end{equation}
\end{theorem}

We next examine this formula for various special cases, starting with the binary case.

\begin{corollary}
Let $\mathcal{M}$ be the $k$-bit independence model.
The formula (\ref{LucaForm}) specializes to
\begin{equation}\label{polarDeg1}
\delta_{r-1}(\mathcal{M}) \,\,\,=\sum_{s=0}^{k-2^k+1+r} \!\! (-1)^s\binom{k+1-s}{2^k-r}(k-s)!\,2^s\binom{k}{s}.
\end{equation}
\end{corollary}

The polar degrees  in (\ref{polarDeg1}) are
shown for $k \leq 7$ in  Table \ref{tab:dims}.
The indices $r$ with $\delta_{r-1} \neq 0$ range from
$ \,{\rm codim}(\mathcal{M}) = 2^k-1-k\,$
to $\,{\rm dim}(\mathcal{M}^*) = 2^k-1$.
  For the sake of the table's layout, we shift the indices so that the row labeled with $0$ contains $\,\delta_{\rm{codim}(\mathcal{M})-1} = {\rm degree}(\mathcal{M}) = k! $.
  The dual variety $\mathcal{M}^*$ is a hypersurface of
  degree $\delta_{2^k-2}$ known 
  as the {\em hyperdeterminant} of  format $2^k$.
  For instance, for $k=3$, this hypersurface in $ \mathbb{P}^7$  is the $2 \times 2 \times 2$-hyperdeterminant which 
has degree four.
  
\begin{small}
\begin{table}[h!]\label{kBitPolar}
	\centering
	\renewcommand*{\arraystretch}{1.2}
	\begin{tabular}{| c | c | c | c | c| c|c|}
		\hline
		$r-\rm{codim}({\mathcal{M}})$ & $k=2$ & $k=3$ &$k=4$& $k=5$& $k=6$ & $k=7$ \\ \hline		
		0           & 2 & 6  & 24  &120 &720 & 5040 \\ \hline
		1        & 2 & 12&72&480 & 3600 & 30240 \\ \hline
		2        & 2 & 12&96& 840&7920 & 80640 \\ \hline
		3      &   & 4&64&  800&9840 & 124320\\ \hline
		4      &   &   &24& 440& 7440 & 120960 \\ \hline
		5      &  &   &   &   128& 3408 & 75936 \\ \hline
		6       &  &   & & & 880 & 30016 \\ \hline		
		7       &  &   & & &  & 6816 \\ \hline
	\end{tabular}
	\vspace{5pt}
	\caption{The polar degrees $\delta_{r-1}(\mathcal{M})$ of the $k$-bit independence model for $k\leq 7$.} \label{tab:dims} 
\end{table}
\end{small}

We next discuss the independence models
$(m_1, m_2)$ for two random variables. These are the classical
contingency tables of format $m_1 \times m_2$.
Here, $n=m_1m_2$ and ${\bf m}=m_1+m_2-2$. 
The ${\bf m}$-dimensional Segre variety $\mathcal{M} = \mathbb{P}^{m_1-1} \times \mathbb{P}^{m_2-1} 
\subset \mathbb{P}^{n-1}$ consists of $m_1 \times m_2$ matrices of rank~one.

\begin{corollary} 
The Segre variety of  $m_1 \times m_2$ matrices of rank one has the polar degrees
\begin{equation}
\label{eq:matrixdelta}
\delta_{r-1}(\mathcal{M}) \,\,\,=\,\, \sum_{s=0}^{{\bf m}-n+1+r}(-1)^s\binom{{\bf m}-s+1}{n-r}({\bf m}-s)!\left(\sum_{i+j=s} \frac{\binom{m_1}{i}}{(m_1-1-i)!}\cdot\frac{\binom{m_2}{j}}{(m_2-1-j)!}  \right).
\end{equation}
\end{corollary}

\begin{small}
\begin{table}[h!]
\label{nAryMayModel}
	\footnotesize
	\centering
	\renewcommand*{\arraystretch}{1.2}
	\begin{tabular}{| c | c | c | c | c| c| c| c| c| c|c| c|}
		\hline
		$r-\rm{codim}({\mathcal{M})}$	 & $(2,3)$ & $(2,4)$ & $(2,5)$&  $(2,6)$ & $(3,3)$& $(3,4)$ & $(3,5)$ & $(3,6)$ & $(4,4)$ & $(4,5)$ & $(4,6)$ \\  
		 \hline 		
		0                       & 3 & 4& 5 & 6& 6  & 10  & 15  &21 & 20 & 35 & 56 \\ \hline
		1                        & 4 &  6& 8& 10&12 & 24  & 40  &60 & 60& 120 & 210 \\ \hline
		2                      & 3 &  4& 5& 6 &12 & 27  & 48  & 75 & 84& 190 & 360 \\ \hline
		3                     &  & & &&6 & 16 & 30  & 48 & 68& 176 & 360 \\ \hline
		4                     & & & & &3 & 6  & 10  &15 & 36& 105 & 228 \\ \hline
		5                      &  &   & &&  &   &   & &  12& 40 & 90 \\ \hline
		6                     &  &  & & & &   &   & &  4& 10 &  20\\ \hline
	\end{tabular}
	\vspace{5pt}
	\caption{The polar degrees $\delta_{r-1}(\mathcal{M})$ of the independence model $(m_1,m_2)$.} 
\label{tab:dims2}
	\end{table}
\end{small}

The polar degrees (\ref{eq:matrixdelta}) are shown in Table \ref{tab:dims2},
with the labeling convention as in Table \ref{tab:dims}.
We now apply the discussion of polar degrees to our optimization problem
for independence models.  Given a fixed model $\mathcal{M}$, the equality in
Theorem~\ref{thm:algdeg} holds only when the data 
$(\ell,L)$ in (\ref{eq:generalopt2}) is generic. However, for the
Wasserstein distance problem stated in (\ref{eq:generalopt}),
the linear space $L=L_F$ and the linear functional $\ell= \ell_F$
are very specific. They depend on the Lipschitz polytope $P_d$
and the type $F$ of the optimal solution $\nu^*$. 
For such specific scenarios, we only get an inequality.

\begin{proposition} \label{prop:upperbound}
Consider the distance optimization problem (\ref{eq:generalopt})  for the  independence model  $(({m_1})_{d_1},\dots , ({m_k})_{d_k})$ on
a given face $F$ of the Wasserstein ball $P_d^*$.
The degree of the optimal solution $\nu^*$ as an algebraic function of
the data $\mu$ is bounded above
by the polar degree $\delta_{r-1}$~in~(\ref{LucaForm}).
\end{proposition}

\begin{proof}
This follows from Theorem~\ref{thm:algdeg}. The upper bound relies on
general principles of algebraic geometry.
Namely, the graph of the map $\mu \mapsto \nu^*(\mu)$ is an irreducible
variety, and we study its degree
over $\mu$.  The map depends on the parameters $(\ell,L)$.
When the coordinates of $L$ and $\ell$ are independent transcendentals
then the algebraic degree is the polar degree $\delta_{r-1}$.
That algebraic degree can only go down when these coordinates take on special 
values in the real numbers.
This semi-continuity argument is valid for most polynomial optimization problems. 
It is~used tacitly for
Euclidean distance optimization in \cite[\S 2]{DHOST} and for
semidefinite programming in \cite[\S 3]{NRS}.
\end{proof}

We now study the drop in algebraic degree for the four models 
 in Example~\ref{ex: running_example}.
In the language of algebraic geometry, our four models are
the Segre threefold $ \mathbb{P}^1 \times \mathbb{P}^1 \times \mathbb{P}^1$
in $\mathbb{P}^7$, the variety $\mathbb{P}^2 \times \mathbb{P}^2 $
of rank one $3 \times 3$ matrices in $\mathbb{P}^8$,
 the rational normal curve $\mathbb{P}^1$ in 
$\mathbb{P}^6 = \mathbb{P}({\rm Sym}_6(\mathbb{R}^2))$, and the
Segre-Veronese surface $\mathbb{P}^1 \times \mathbb{P}^1$ in 
$\mathbb{P}^5 = \mathbb{P}({\rm Sym}_2(\mathbb{R}^2) \times 
{\rm Sym}_1(\mathbb{R}^2) )$. The underlying finite metrics $d$ are specified in the
fourth column of Table \ref{table_experiments_combinatorics}.
The fifth column records the combinatorial complexity of our optimization problem,
while the algebraic complexity is recorded in Table \ref{table_experiments_algebra}.

\begin{table}[h!]
	\begin{center}
		\vspace{-0in}
		\renewcommand*{\arraystretch}{1.2}
		\begin{tabular}{ | r | l | l | l | }
			\hline
			$\mathcal{M}$ & Polar degrees  & Maximal degree & Average degree\\ \hline
			$(2,2,2)$ & $(0,0,0,6,12,12,4) $ & $(0,0,0,4,12,6,0)$ & $(0,0,0,2.138,6.382,3.8,0)$   \\
			$(3,3)$ & $(0,0,0,6,12,12,6,3)$ & $(0,0,0,2,8,6,6,0)$ & $(0,0,0,1.093,3.100,4.471,6.0,0)$\\
			$(2_6) $ & $(0,0,0,0,6,10)$ & $(0,0,0,0,6,5)$& $(0,0,0,0,6,5)$\\ 
			$(2_2,2)$ & $(0,0,4,6,4)$ & $(0,0,3,5,2)$ & $(0,0,2.293,3.822,2.0)$\\ \hline
		\end{tabular}
		\vspace{-0.2in}
	\end{center}
	\caption{The algebraic degrees of the problem (\ref{eq:ourproblem})
		for the four models in \Cref{ex: running_example}.}
	\label{table_experiments_algebra}
\end{table}	

\smallskip

The second column in  Table \ref{table_experiments_algebra} gives
the vector $(\delta_0,\delta_1,\ldots,\delta_{n-2})$
of polar degrees for the model $\mathcal{M}$ under consideration.
The third and fourth column are results of our computations.
For each model, we take $1000$ uniform samples $\mu $ with
rational coordinates from the simplex $\Delta_{n-1}$,
and we solve the optimization problem (\ref{eq:ourproblem}) using the
methods described in Section \ref{sec6}. The output is an exact
representation of the optimal solution $\nu^*$. This includes the optimal face $F$ 
that specifies $\nu^*$, along with its maximal ideal in the polynomial ring over
  the field $\mathbb{Q}$
of rational numbers. The algebraic degree of the optimal solution $\nu^*$ is 
computed as the
number of complex zeros of that maximal ideal. This number
is bounded above by the polar degree, as seen in Proposition \ref{prop:upperbound}.

The third and fourth column in  Table \ref{table_experiments_algebra}
reports on the algebraic degree of $\nu^*$ in our experiments.
It shows the maximum and the average of the degrees 
found in the $1000$ computations. That maximum is
bounded above by the polar degree. Equality holds in some cases.
For example,  for the $3$-bit model $(2,2,2)$ we have $\delta_3 = 6$, 
corresponding to $P_d^*$ touching $\mathcal{M}$ at a $3$-face $F$,
but the maximum degree we observed was $4$, with an average
degree of $2.138$. For $4$-faces $F$, we have
$\delta_4 = 12$, and this was indeed attained in some of our experiments.
The average was $6.382$.

\section{Algorithms and Experiments}            \label{sec6}

We now report on computational experiments. These are carried out in three stages:
(1) combinatorial preprocessing, (2) numerical optimization, and (3) algebraic postprocessing.
Our object of interest is a model $\mathcal{M}$ in the simplex $\Delta_{n-1}$,
typically one of the independence models $(({m_1})_{d_1},\dots , ({m_k})_{d_k})$
where $n = \prod_{i=1}^k \binom{m_i+d_i-1}{d_i}$.
The state space $[n]$ is given the structure of a metric space
by a symmetric $n \times n$ matrix $d = (d_{ij})$. This matrix defines
the Lipschitz polytope $P_d$ and its dual, the Wasserstein ball $P_d^*$.
Our first algorithm computes these combinatorial objects.

\vspace{10pt}
\begin{algorithm}[H]\label{algorithm1}
	\KwIn{ An $n\times n$ symmetric matrix $d=(d_{ij})$.}
	\KwOut{A description of all facets $F$ of the Wasserstein ball $P_d^*$.}
	\KwSty{Step 1:} 
	From the description in Section \ref{sec4}, 
	find all vertices of the Lipschitz polytope~$P_d$. These vertices are the inner normal vectors $\ell_F$ to the facets $F$ of $P_d^*$.
	Store them. \\
	\KwSty{Step 2: } Determine an inequality description of the cone $C_F$ over
	each facet $F$. \\
	\KwSty{Return:} The list of pairs $(\ell_F,C_F)$,
	 one for each vertex of the Lipschitz polytope $P_d$.
		\caption{Combinatorial preprocessing}
	\label{alg: algorithm1}
\end{algorithm}
\vspace{10pt}

In our experiments, we use the software {\tt Polymake} \cite{polymake} for
running Algorithm \ref{alg: algorithm1}. Note that Step 1 is 
a challenging calculation. It remains an open problem to characterize combinatorially the incidence structure of other Lipschitz polytopes in the same spirit as \Cref{lem:rhombic}. We carried out this preprocessing
for a range of smaller models including those four featured in Example~\ref{ex: running_example}.

Our next algorithm solves the optimization problem in (\ref{eq:ourproblem}).
This is done by examining each facet $F$ of the Wasserstein ball.
The problem is precisely that in (\ref{eq:generalopt}) but with the linear space
$L_F$ now replaced by the convex cone $C_F$ that is spanned by $F$.

\vspace{5pt}
\begin{algorithm}[H]\label{alg: algorithm2}
	\KwIn{Model $\mathcal{M}$ and a point $\mu$ in  the simplex $ \Delta_{n-1}$;
		complete output from Algorithm~\ref{alg: algorithm1}.	}
	\KwOut{The optimal solution $\nu^*$ in (\ref{eq:ourproblem}) along with its type $G$.}
	\KwSty{Step 1: } \For{\normalfont{each facet $F$ of the Wasserstein ball $P_d^*$} }{
		\KwSty{Step 1.1: } 	
		Apply global optimization methods
		to identify a solution $\nu^* \in \mathcal{M}$ of 
		\begin{equation*}
		\label{eq:generalopt3}
		\hbox{{\rm minimize} $\,\ell_F = \ell_F(\nu)\,$ subject to $\,\nu \in (\mu + C_F) \cap \mathcal{M}$.} 
		\end{equation*} \\
		\KwSty{Step 1.2: } 	Identify the unique face $G$ of $F$ whose span has $\nu^*$ in its relative interior. \\
		\KwSty{Step 1.3: } Find a basis of vectors $e_i-e_j \in C_G$ for 
		the linear space $L_G$ spanned by~$G$. \\
		\KwSty{Step 1.4: } Store the optimal solution $\nu^*$ and 
		a basis for the linear subspace $L_G$ of $\mathbb{R}^n$.}
	\KwSty{Step 2: } Among candidate solutions found in Step 1, identify
	the solution $\nu^*$ for which the Wasserstein distance $W_d(\mu,\nu^*)$ 
	to the  data point $\mu$ is smallest. Record its type~$G$.\\
	\KwSty{Return:} The optimal solution $\nu^*$, its associated linear space $L_G$, and
	the facet normal $\ell_G$.
	\caption{Numerical optimization}
	\label{alg: algorithm3}
\end{algorithm}
\vspace{5pt}

We use the software {\tt SCIP} \cite{SCIP} for running Algorithm~\ref{alg: algorithm2}. {\tt SCIP} employs sophisticated branch-and-cut strategies to solve constrained polynomial optimization problems via LP relaxation. We make use of the Python interface in {\tt SCIP} to implement Algorithm~\ref{alg: algorithm3} in a single environment. 

The virtue of Algorithm~\ref{alg: algorithm3} is that it is guaranteed to find
the global optimum for our problem~(\ref{eq:ourproblem}). Moreover, it furnishes
an identification of the combinatorial type. This serves as the input
to the symbolic computation in Algorithm~\ref{alg: algorithm4}.
The drawback of Algorithm~\ref{alg: algorithm3} is that it
requires reprocessing that is prohibitive for
larger models. We will return to this point later.

\vspace{5pt}
\begin{algorithm}[H]\label{algorithm3}
	\KwIn{ The optimal solution $(\nu^*,G)$ to (\ref{eq:ourproblem}) in the form
	found by Algorithm~\ref{alg: algorithm3}.}
	\KwOut{The maximal ideal in the polynomial ring  
	$\mathbb{Q}[\nu_1,\ldots,\nu_n]$ which has the zero $\nu^*$.}
	\KwSty{Step 1:} Use Lagrange multipliers to give polynomial equations that
	characterize the critical points of the linear function $\,\ell_F\,$ on the subvariety
	$\,(\mu+L_G) \cap \mathcal{M}\,$ in $\,\mathbb{R}^n$. \\
	\KwSty{Step 2: } Eliminate all variables representing Lagrange multipliers from the ideal
	in~Step~1. \\
\KwSty{Step 3: } The ideal from Step 2 is in $\mathbb{Q}[\nu_1,\ldots,\nu_n]$. 
If this ideal is maximal then call it $M$. \\
	\KwSty{Step 4: }  If not, remove extraneous primary components to get 
	the maximal ideal $M$ of $\nu^*$. \\
\KwSty{Step 5: }  Determine the degree of $\nu^*$, which is the
 dimension of $\mathbb{Q}[\nu_1,\ldots,\nu_n]/M$
 over $\mathbb{Q}$. \\
		\KwSty{Return:} Output generators for the ideal $M$ along with the degree
		found in Step 5.
				\caption{Algebraic postprocessing}
		\label{alg: algorithm4}
\end{algorithm}
\vspace{3pt}

We run
Algorithm \ref{alg: algorithm4} with the computer algebra system {\tt Macaulay2} \cite{M2}.
 Steps 2 and 4 are the result of standard Gr\"obner basis calculations.
We illustrate the entire pipeline with an example.

\begin{example}
The following matrices are points in the probability simplex $\Delta_8$
for the model~$(3,3)$:
$$ 
\mu \,=\, 
\begin{small} \frac{1}{100} \begin{bmatrix}
                                             2  &    3 &     5 \\
                                             7  &  11  &  13 \\
                                             17 &   19 &   23 \end{bmatrix} \end{small}
                                             \, ,
\quad                                             
\nu^* \,=\, \begin{small} \frac{1}{4600} \begin{bmatrix}
124    & 152  &     184 \\
403   & 494  &   598 \\
713  &  874  &  1058
\end{bmatrix}            \end{small}   \, ,                              
\quad
\hat \nu \,=\, \begin{small} \frac{1}{10000} \begin{bmatrix}
260   &  330 &    410 \\
806  &  1023 &   1271 \\
1534 &   1947 &   2419
\end{bmatrix}. \end{small}
$$                                             
Algorithm~\ref{alg: algorithm2} computes
the optimal solution $\nu^*$ along with its type $G$. This 
  face of the $8$-dimensional
Wasserstein ball $P_d^*$ is the tetrahedron
$ G\,\,=\,\,\conv\{e_1-e_2,e_2-e_3,e_4-e_5,e_4-e_7\}$.
The four vertices span the linear space $L_G$.
A facet $F$ containing $G$ is defined by the normal vector 
$\ell_F=(2,1,0,1,0,1,0,-1,0)$.
While the corresponding polar degree
$\delta_3$ equals $6$, Table~\ref{table_experiments_algebra} shows
that all solutions observed for this type have
algebraic degree $1$ or $2$, with average  $1.093$.
Indeed, the entries of the matrix $\nu^*$ are rational numbers, so 
 the algebraic degree is $1$. 
The optimal Wasserstein distance is the rational number
$\,W_d(\mu,\nu^*) =\langle \ell_F,\mu-\nu^*\rangle =  159/4600 = 0.034565217.... $

The rightmost matrix $\hat \nu$ also has rank one.
It lies in the model, just like $\nu^*$.
This matrix is the maximum likelihood estimate for $\mu$, so it minimizes
the Kullback-Leibler distance to the model.
Its Wasserstein distance to the data $\mu$ equals
$\, W_d(\mu,\hat \nu) = 32/625 = 0.0512$. In the experiments recorded in \Cref{table_experiments_new},
 the type $G$ of the solution $\nu^*$ has dimension $3$  for the $65.7\%$ of the samples~$\mu$.

We now consider another data point, obtained by permuting the coordinates  used above:
$$ 
\mu \,=\, \frac{1}{100} \begin{bmatrix}
                                             11  &    2 &     5 \\
                                             3  &  13  &     7 \\
                                             17 &   19 &   23 \end{bmatrix}\, ,
\quad                                             
\nu^* \, = \, \begin{bmatrix}
\nu_1 & \nu_2 & \nu_3 \\
\nu_4 & \nu_5 & \nu_6 \\
\nu_7 & \nu_8 & \nu_9 \end{bmatrix} \, = \,
\begin{bmatrix}
0.037183 &  0.041558  & 0.050303 \\
0.080956 &  0.090480 &  0.109525 \\
0.17 & 0.19 &                   0.229995 
\end{bmatrix}.
$$
Here Algorithm~\ref{alg: algorithm2} identifies the
solution $\nu^*$ above, together with the $4$-dimensional type 
\[
	G \,\,=\,\, \conv\{e_2-e_1,e_3-e_2,e_4-e_1,e_6-e_5,e_6-e_9\}.
\]
The optimal value, $W_d(\mu,\nu^*) = 0.112645$,
has algebraic degree $4$, so it can be written
in radicals over $\mathbb{Q}$.
The relevant polar degree is $\delta_4 = 12$. The largest observed
degree is $8$, as seen in Table~\ref{table_experiments_algebra}.
The exact representation of the solution $\nu^*$ is the 
maximal ideal in $\mathbb{Q}[\nu_1,\ldots,\nu_9]$ generated by
\small{$$ \begin{matrix}
5631250000 \nu_1^4-18245250000 \nu_1^3-3922376250 \nu_1^2-121856850 \nu_1+9002061,\, \\
17 \nu_2-19 \nu_1,\quad 100 \nu_7-17,\quad
100 \nu_8-19,\, \\
10489919785 \nu_3\,+\,954632025000 \nu_1^3-3208398380500 \nu_1^2-261822911570 \nu_1+11757750732
, \, \\
12341082100 \nu_4 \,-\,1123096500000 \nu_1^3+3774586330000 \nu_1^2+334161011000 \nu_1
-16424275161 ,\, \\
209798395700 \nu_5\,-\,21338833500000 \nu_1^3+71717140270000 \nu_1^2+6349059209000 \nu_1
-312061228059,\, \\
104899197850 \nu_6\,+\,23173044250000 \nu_1^3-77993197677500 \nu_1^2-6429496583150 \nu_1
+285451958883,\, \\
104899197850 \nu_9\,-\,12503627500000 \nu_1^3+42134627542500 \nu_1^2+3254966978650 \nu_1
-174527999929 .\\
\end{matrix}
$$}
This Gr\"obner basis in triangular form is the output of Algorithm~\ref{alg: algorithm4}.
Two entries of $\nu^*$ are rational.
\end{example}

\vspace{5pt}
\begin{table}[h!]
	\begin{center}
		\vspace{-0.14in}
		\begin{tabular}{ | r | l | l |  l |  l | }
			\hline
			$\mathcal{M}$ & $d$ & $\dim(\mathcal{M})$ & $\#$ facets of $B$ & avg $\#$ feasible probs. \\ \hline 
			$(2,2)$ & $L_0$ & 2 & $6$ & 5.000 \\    \hline
			$(2,2,2)$ & $L_0$ & 3 & $38 $& 23.734  \\ \hline
			$(2,3)$ &  $L_0$ & 3 & $54$ & 30.000 \\ \hline
			$(2,3)$ & $L_1$ & 3 & $18$ & 12.645  \\ \hline
			$(3,3)$ &  $L_0$ & 4 & $534$ & 162.307 \\ \hline
			$(3,3)$ & $L_1$ & 4 & $82$ & 40.626 \\ \hline
			$(2,4)$ &  $L_0$ & 4 & $282$ & 110.165 \\ \hline
			$(2,4)$ & $L_1$ & 4 & $54$ & 32.223 \\ \hline
			$(2_3)$ & $L_1$ & 1 & $8$ & 4.000 \\ \hline
			$(2_3)$& di & 1 & $14$ & 5.182\\ \hline
			$(2_2,2)$ & $L_1$ & 2 & $18$ & 8.604\\ \hline
			$(2_2,2)$  & di & 2 & $62$ & 24.618 \\ \hline
			$(3_2)$ & di & 2 & $62$ & 24.365 \\ \hline
			$(2_4)$ & $L_1$ & 1 & $16$ & 5.000 \\ \hline
			$(2_4)$  & di & 1 & $30$ & 8.690 \\ \hline
		\end{tabular}
	\end{center}
	\caption{The number of feasible optimization problems for a uniform sample of 1000 points.}
	\label{table_experiments_new_first}
\end{table}

Using our three algorithms, we ran experiments
 on various models with $1000$ uniformly sampled data points $\mu$.
 The first question we addressed:
  \emph{For a given data point $\mu$, how many~of the polynomial optimization problems in Step 1.1 of Algorithm~\ref{alg: algorithm3} are feasible? }
In geometric~terms: for how many facets $F$ of the ball $P_d^*$
does the cone $\mu + C_F$ intersect the model?
A bound for this number could be used to reduce the number of optimization problems in Step 1 of Algorithm~\ref{alg: algorithm2}.  We report the average number of feasible problems for several models and metrics in \Cref{table_experiments_new_first}. We observe that
different metrics for the same model can produce quantitatively different results.

Our second question is: \emph{What is the distribution of the dimension of the type $G$
for $\mu\in\Delta_{n-1}$?} The output of Algorithm~\ref{alg: algorithm2} contains that information. 
We display it in \Cref{table_experiments_new}  for the  same models and metrics
 as in Table \ref{table_experiments_new_first}.
 For some models unexpected intersections happened. For example, the second row shows
 that for 1 of the 1000 random points the optimal type was a $2$-dimensional face, even though generically a $3$-dimensional linear space does not intersect a model with codimension~$4$. This is due to numerical imprecision.
In \Cref{thm: 3bits}, we studied the $2$-bit model, and we 
 saw that the intersection of the Wasserstein ball and the model is either an edge or a vertex. The first row of \Cref{table_experiments_new} shows that, on a uniform sample of 1000 points in the
tetrahedron $\Delta_3$, in
  roughly  $31\%$ of the cases the intersection lies in the interior of an edge. Looking at \Cref{fig:partition12}, this indicates the fraction of volume enclosed between the red surfaces and the edges of $\Delta_3$ they cover.\\

\begin{table}[h!]
	\footnotesize
	\begin{center}
		\vspace{-0.14in}
		\begin{tabular}{ | r | l |  l |  l |  l | l | l | l | l | l |}
			\cline{4-10}  
			\multicolumn{3}{c|}{}&\multicolumn{7}{c| }{$\%$ of opt. solutions of $\dim(\text{type})=i$} \\ \hline
			$\mathcal{M}$ & $d$ & $f$-vector & 0 & 1 & 2 & 3 & 4 & 5 & 6 \\ \hline 
			$(2,2)$ & $L_0$ & $(8,12,6)$ & 68.6 & 31.4 & 0 & - & - & - & - \\    \hline
			$\! (2,2,2)$ & $L_0$ & $(24, 192, 652, 1062, 848, 306, 38) $& 0 & 0 & 0.1 & 70.9 & 27.5 & 1.5 & 0 \\ \hline
			$(2,3)$ &  $L_0$ & $(18,96,200,174,54)$ & 0 & 64.1 & 18.7 & 17.2 & 0 & - & - \\ \hline
			$(2,3)$ & $L_1$ & $(14,60,102,72,18)$ & 0 & 76.7 & 17.4 & 5.9 & 0 & - & - \\ \hline
			$(3,3)$ &  $L_0$ & (36{,}468{,}2730{,}8010{,}12468{,}10200{,}3978{,}534)  & 0 & 0 & 0.1 & 58.3 & 28.2 & 4.6 & 8.8 \\ \hline
			$(3,3)$ & $L_1$ & $(24, 216, 960, 2298, 3048, 2172, 736, 82)$ & 0 & 0 & 0 & 65.7 & 27.8 & 5.1 & 1.4  \\ \hline
			$(2,4)$ &  $L_0$ & $(32, 336, 1464, 3042, 3168, 1566, 282)$ & 0  & 0.1 & 55.1 & 14.6 & 25.8 & 4.4 & 0\\ \hline
			$(2,4)$ & $L_1$ & $(20, 144, 486, 846, 774, 342, 54)$ & 0 & 0 & 75.3 & 16.5 & 8.2  & 0 & 0 \\ \hline
			$(2_3)$& $L_1$ & $(6,12,8)$ & 0 & 98.3 & 1.7 & - & - & - & -  \\ \hline
			$(2_3)$& di & $(12,24,14)$ & 0.2 & 96.7 & 3.1 & - & - & - & - \\ \hline
			$(2_2,2)$ & $L_1$ & (14,60,102,72,18) & 0 & 0 & 67.6 & 27.5 & 4.9 & - & -\\ \hline
			$(2_2,2)$  & di & $(30,120,210,180,62)$ & 0 & 0.2 & 81.9 & 16.8 & 1.1 & - & - \\ \hline
			$(3_2)$ & di & $(30,120,210,180,62)$ & 0 & 0.2 & 83.1 & 16.0 & 0.7 & - & - \\ \hline
			$(2_4)$ & $L_1$ & $(8,24,32,16)$ & 0 & 0.1 & 98.3 & 1.6 & - & - & -  \\ \hline
			$(2_4)$  & di & $(20,60,70,30)$ & 0 & 0 & 96.9 & 3.1 & - & - & - \\ \hline
		\end{tabular}
	\end{center}
	\caption{Distribution of types among optimal solutions for a
	 uniform sample of 1000 points.}
	\label{table_experiments_new}
\end{table}

In this article we studied the Wasserstein distance problem for discrete statistical models, with emphasis on the combinatorics, algebra and geometry of independence models. 
The theoretical results we obtained here constitute the foundation for a class of iterative algorithms that can be applied to larger models. 
We shall develop such algorithms and their implementation in a forthcoming project, with a view towards concrete applications of our methods in data science. 

\section*{Acknowledgment}

Asgar Jamneshan was supported by DFG-research fellowship AJ 2512/3-1. Guido Mont\'ufar
 acknowledges support from the European Research Council (ERC) under the European Union’s Horizon 2020 research and innovation programme (grant no 757983). 
We thank Felipe Serrano for helping us with the software {\tt SCIP}. 

\bibliographystyle{elsarticle-harv}
\bibliography{bibliography}

%
\end{document}